\documentclass[reqno]{amsart}
\usepackage{amsmath, amsfonts, amsthm, amssymb, multicol,mathrsfs}
\usepackage{graphicx}
\graphicspath{ {./Images/} }
\usepackage[margin=1.4in, centering]{geometry}
\usepackage{float}
\usepackage{xcolor}
\usepackage{verbatim}
\usepackage[english]{babel}
\usepackage{comment}
\usepackage{hyperref}
\hypersetup{colorlinks}
\hypersetup{
    colorlinks=true,
    linkcolor=blue,
    filecolor=magenta,      
    urlcolor=cyan,
    citecolor=blue,
}

\usepackage{tikz-cd}
\usepackage{tikz-3dplot}
\numberwithin{equation}{section}

\newtheorem{thm}{Theorem}[section]
\newtheorem{lma}[thm]{Lemma}
\newtheorem{cor}[thm]{Corollary}
\newtheorem{defn}[thm]{Definition}

\newtheorem{prop}[thm]{Proposition}

\newtheorem{ques}[thm]{Question}

\newtheorem{example}{Example}

\renewcommand{\epsilon}{\varepsilon}

\newcommand{\R}{\mathbb{R}}

\renewcommand{\ge}{\geqslant}

\renewcommand{\geq}{\geqslant}
\renewcommand{\leq}{\leqslant}

\newcommand{\hdim}{\dim_{\mathrm{H}}  }

\DeclareMathOperator{\Vol}{Vol}
\DeclareMathOperator{\rank}{rank}
\DeclareMathOperator{\M}{M}
\DeclareMathOperator{\dir}{dir}

\newcommand{\esssup}{\mathop{\operatornamewithlimits{ess-sup}}}
\DeclareMathOperator{\aff}{aff}
\DeclareMathOperator{\dist}{dist}
\DeclareMathOperator{\G}{\mathcal{G}}
\DeclareMathOperator{\A}{\mathcal{A}}

\DeclareMathOperator{\Proj}{Proj}

\author{Jos\'{e} Gaitan Montejo}
\address[J. Gaitan]{Department of Mathematics, Virginia Tech, Blacksburg, VA}
\email{jogaitan@vt.edu}

\author{Eyvindur Ari Palsson}
\address[E. Palsson]{Department of Mathematics, Virginia Tech, Blacksburg, VA}
\email{palsson@vt.edu}

\title{On volumes of simplices in intermediate dimensions}

\begin{document}

\begin{abstract}
A variant of the Falconer distance problem asks for fixed $k\geq 1$ and $d\geq k+1$, how large does the Hausdorff dimension of a Borel set $E\subset\mathbb{R}^d$ need to be to guarantee that there exist $x_0,\ldots,x_{k}\in E$ such that $\Vol_{k+1}^{(x_0,\ldots,x_{k})}(E) = \lbrace \Vol_{k+1}(x_0,\ldots,x_{k},x_{k+1}) : x_{k+1}\in E \rbrace$ has positive Lebesgue measure. Here $\Vol_{k+1}(x_0,\ldots,x_{k},x_{k+1})$ denotes the $k+1$-volume of the $k+1$ simplex formed by $x_0,\ldots,x_{k},x_{k+1}$. Recently, Shmerkin and Yavicoli established a sharp dimensional threshold $k$ in the case when $d=k+1$. In this paper we extend their result to $k+1 \leq d \leq 2k$ and obtain a non-trivial dimensional threshold $d-k$ when $d>2k$.

The result is motivated by ideas from Shmerkin and Yavicoli. A crucial part of the argument is an application of work by Bright, Ortiz and Zakharov on a continuum Beck-type theorem for hyperplanes as well as classic results of Marstrand on orthogonal projection theorems. In addition, we investigate a more elementary approach under a Fubini property condition for Hausdorff dimension as introduced in the work of H\'{e}ra, Keleti and M\'{a}th\'{e}.
\end{abstract}

\maketitle

\section{Introduction}

\vskip.125in 

The Falconer distance problem, a continuous analogue of the celebrated Erd\H{o}s distance problem, asks how large the Hausdorff dimension of a Borel set $E\subset\mathbb{R}^d$, $d\geq 2$, needs to be to guarantee that the Lebesgue measure of the distance set $$\Delta(E) := \lbrace |x-y| : x,y\in E\rbrace$$ is positive. Falconer \cite{Falc1985} showed that the threshold $\frac{d}{2}$ was necessary and conjectured it was also sufficient. Much progress has been made on this problem in the last decade with the current best known thresholds being $\frac{5}{4}$ when $d=2$ and $\frac{d}{2}+\frac{1}{4}-\frac{1}{8d+4}$ when $d\geq 3$ \cite{GIOW20,DORZ,DORZ24}. These papers in fact show a stronger result that with the same thresholds there exists $x\in E$ such that the pinned distance set $$\Delta^{x}(E) := \lbrace |x-y| : y\in E\rbrace$$ has positive Lebesgue measure.

Replacing distance with another geometric quantity gives rise to many variants of the Falconer distance problem. In this paper we consider the case of volumes of simplices. For $k\geq 1$, points $x_0, \ldots, x_{k+1}$ in $\mathbb{R}^d$, $d\geq k+1$, determine a possibly degenerate $(k+1)$-simplex and we denote the $(k+1)$-volume of that simplex by $\Vol_{k+1}(x_0,\ldots,x_{k+1})$. A natural analogue of the Falconer distance problem then asks how large the Hausdorff dimension of a Borel set $E\subset\mathbb{R}^d$, $d\geq k+1$, needs to be to guarantee that the Lebesgue measure of the volume set 
$$\Vol_{k+1}(E) := \lbrace \Vol_{k+1}(x_0,\ldots,x_{k+1}) : x_0,\ldots,x_{k+1}\in E \rbrace$$
is positive. The restriction $d \geq k+1$ arises naturally, for otherwise all simplices are degenerate and the only volume that shows up is $0$.

The first result in this direction was established by Erdo\u{g}an, Hart and Iosevich \cite{EHI13}, who obtained the threshold $\frac{3}{2}$ for areas of triangles in $\mathbb{R}^2$ pinned at the origin. Greenleaf, Iosevich and Mourgoglou \cite{GIM15} extended that result to volumes of tetrahedra pinned at the origin in $\mathbb{R}^3$ with the dimensional threshold $\frac{8}{3}$. A case of particular interest is when $d=k+1$ and in this setting Grafakos, Greenleaf, Iosevich and the second listed author of this paper obtained the dimensional thresholds $d-1+\frac{1}{2d}$ if $d$ is even and $d-1+\frac{1}{2(d-1)}$ if $d$ is odd, for the $d$-simplex in $\mathbb{R}^d$ \cite{GGIP}. It is easy to see by considering a hyperplane, that the threshold $d-1$ is necessary and in \cite{GGIP} it was conjectured this was also sufficient. A number of extensions have followed. Greenleaf, Iosevich and Taylor obtained non-empty interior for areas of triangles at a threshold $\frac{5}{3}$ and a strongly pinned non-empty interior result at a dimensional threshold $d-1+\frac{1}{d}$ when $k+1=d\geq 3$ \cite{GIT22}. McDonald \cite{McDonald21} as well as Galo and McDonald \cite{GMD22} extended these results respectively to area-types and volume-types, which are generalizations that make sense of the question when $d<k+1$. In related work to the area-types, Greenleaf, Iosevich and Taylor obtained non-empty interior results for tuples of areas \cite{GIT24} and even extended their results to trees of areas recently \cite{GIT25}.

In the much studied case, when $d=k+1$, where the ambient space matches the size of the simplex, very recently Shmerkin and Yavicoli \cite{SY25} confirmed the conjecture in \cite{GGIP} and obtained a very strong pinned result at threshold $d-1$. More precisely, they showed for a Borel set $E\subset\mathbb{R}^d$ with $\hdim(E)>d-1$ then there exist pins $x_0,\ldots,x_{d-1}\in E$ such that
$$\Vol_{d}^{(x_0,\ldots,x_{d-1})}(E):=\lbrace \Vol_{d}(x_0,\ldots,x_d) : x_d\in E \rbrace$$
has positive Lebesgue measure. In other words, when $d=k+1$, if $\hdim(E)>k$, there exist pins $x_0,\ldots,x_{k}\in E$ such that $\Vol_{k+1}^{(x_0,\ldots,x_{k})}(E)$ has positive Lebesgue measure. The dimensional threshold $\hdim(E)>k$ is necessary, as shown by the hyperplane example, and remains necessary in higher dimensions $d>k+1$. Through an iterative hyperplane slicing argument, one can extend the arguments of Shmerkin and Yavicoli to a dimensional threshold $d-1$ in the higher dimensional case, but the question remains  whether that threshold can be lowered down to $k$.

The statement of \cite[Theorem~1.1]{SY25} claims the threshold $\hdim(E)>k$ for all ambient dimensions $d\geq k+1$. However, the proof given in \cite{SY25} does not yield the claimed result in the higher dimension case, $d>k+1$. The problem is that the proof controls distances to an affine hyperplane $H$ in which the base of the simplex lies. If $ W=\aff\{x_0,\ldots,x_k\}$ is the affine $k$-plane spanned by the simplex base, then the relevant identity is
\begin{equation*}
 \Vol_{k+1}(x_0,\ldots,x_k,y) =\frac{\Vol_k(x_0,\ldots,x_k)}{k+1}\,\dist(y,W).   
\end{equation*}
When $d=k+1$, the given $W$ is itself an affine hyperplane. But when there is a mismatch between the dimensionality of the base and the hyperplane, the information about the set $\{\dist(y,H):y\in E\}$ does not control the values of $\{\dist(y,W):y\in E\}$. Indeed, there could possibly be many distances to the hyperplane, while the distances to the base could be very few, for example, think of a plane containing a line in $\R^3$ and a cylinder centered at the line, the set of distances from the cylinder to the line are a singleton, but to the plane, the  set of distances is a full interval.

The present argument resolves this difficulty by investigating directly the set of values of $\dist(y,W)$ without having to consider first a containing hyperplane. The key idea is to consider the identity
$$\dist(y,W) =\bigl|\Proj_{W_0^{\perp}}(y-p)\bigr|,$$
where $W=W_0+p$ with $p\in W_0^{\perp}$. Therefore, the volume problem is turned into an orthogonal projection problem restricted to a particular family of $(d-k)$-subplanes that are orthogonal complements of the directions of $k$-affine planes spanned by points in the set $E$.

A continuum Beck-type theorem \cite{BOZ} for $k$-planes plays a crucial role in the argument to guarantee the abundance of $(d-k)$-subplane directions that are the projection target spaces of the restricted family described above. The abundance of subplane directions is such, that it surpasses that of the set of exceptional projections.

For fixed $k\geq 1$, our main theorem establishes the dimensional threshold $\hdim(E)>k$ to guarantee the existence of pinned vertices $x_0,\ldots,x_{k}\in E$ such that $\Vol_{k+1}^{(x_0,\ldots,x_{k})}(E)$ has positive Lebesgue measure in the range of ambient dimensions $k+1\leq d \leq 2k$. 

The case $d=k+1$ is given due to \cite{SY25}, but in the range $k+1 < d \leq 2k$,  the results are new and sharp. We also obtain a non-trivial threshold $d-k$ in higher dimensions, which improves the $d-1$ threshold that can be obtained from iteratively applying the codimension one result.

\begin{thm}\label{main}
  Fix $k\geq 1$ and then $d\geq k+1$. Let $E \subseteq \R^d$ be a Borel set  with dimension $\hdim(E)>\max\{k,d-k\}$. Then there exist $x_{0},\ldots,x_{k}\in E$ such that the set
  \[
    \Vol_{k+1}^{(x_0,\ldots,x_{k})}(E) :=\left\{ \Vol_{k+1}(x_0,\ldots,x_k,y): y\in E \right\}
  \]
has positive Lebesgue measure. Note that the condition reduces to $\hdim(E) >k$ if $d\leq 2k$.
\end{thm}
Note that throughout this paper the dimension of the
$(k+1)$-simplex is fixed, while the ambient Euclidean dimension varies. A concrete example is provided by \textit{tetrahedra}.
\begin{example}
A tetrahedron is a $3$-simplex, so $k=2$ is its corresponding value in the notation of Theorem \ref{main}. Thus, for any Borel set $E\subset\R^d$, $d\geq 3$,  Theorem \ref{main} gives the sufficient dimensional threshold $\hdim(E)>\max\{2,d-2\}$ to guarantee the existence of  a pinned triangular base of points $x_{0},x_1,x_2\in E$ such that
\begin{equation*}
\mathcal{L}^1(\Vol_{3}^{(x_0,x_{1},x_2)}(E))>0.    
\end{equation*}
 In the intermediate dimensions regime $3\leq d\leq 2k=4$, the dimensional requirement reduces to $\hdim(E)>2$ and it is sharp. The case $E\subset \R^3$ is the codimension one situation that was already known due to \cite{SY25}, but $E\subset \R^4$ is the first new instance where Theorem \ref{main} applies. 
 For higher dimensions, $d\geq 5$, however, the sufficient condition given by our theorem is $\hdim(E)>d-2$. This threshold arises from our method, but we do not expect it to be sharp.
\end{example}
We would like to remark that $\hdim(E)>\max \{k,d-k\}$ is required in our method, the condition $\hdim(E) >k$ is clear, otherwise, $E$ could be entirely contained on a single $k$-plane.  Whereas $\hdim(E) >d-k$ is used after finding the fixed base of points $x_{0},\ldots,x_{k}\in E$ and projecting $E$ onto the complementary direction of the $k$-affine plane containing the base points. The proof then requires the dimension of $E$ to be greater than the dimension of the projection target. This is a restriction of the present method, and there is no indication that $d-k$ is a necessary threshold for the underlying volumes problem. In fact, after restricting to a suitable affine subspace, the relevant projection target can be smaller.
\subsection{Affine  \texorpdfstring{$k$}{k}-planes spanned by points from a Borel set.} One of the key steps in the argument of \cite[Theorem 1.1] {SY25} is to demonstrate the existence of many affinely independent $k$-planes spanned by $k+1$ points in $E$. They achieve this in the codimension one case by using Marstrand's slicing theorem (Theorem \ref{mar2} equation \eqref{mareq3}) to control the dimension of the intersections of $E$ with affine hyperplanes. They found that in almost every slice it is possible to find at least one collection of affinely independent points $x_0,\dots, x_{d-1}\in E\cap H$, and with abundant complement $\Proj_{H^{\perp}}(E)$. This idea suggests that a natural approach to obtain the result for volumes of $k$-simplices is, therefore, to slice $E$ with lower dimensional affine $k$-planes and hope to simultaneously have abundant intersections and complements. However,  in the case $d>k+1$, considering intersections of hyperplane slices iteratively to try to replicate the strategy from \cite{SY25} does not yield good quantitative bounds on the abundance of affine $k$-planes spanned by $E$.

Recently, a related question was addressed by Bright, Ortiz and Zakharov \cite{BOZ}. Let $\mathcal P^k(E)$ denote the set of $k$-planes spanned by $k+1$ affinely independent points of $E$. The authors obtained positive and optimal results about the dimension of $\mathcal P^k(E)$ under \emph{non-concentration} assumptions on the Borel set $E$. Their main goal was a continuum Beck-type theorem for hyperplanes, so the case of $k=d-1$ was of special interest to them. One of the non-concentration assumptions they considered is to require $E$ to admit an \emph{irreducible} $s$-$Frostman$ measure in the following sense.
\begin{defn}\cite[Definition 1.5]{BOZ}\label{defn: irreducibility of a measure}.
    A measure $\mu$ supported on an affine subspace $V\subset \R^d$ is \textbf{irreducible} in $V$ if $\mu(V') = 0$ for any proper affine subspace $V'\subset V$.
\end{defn}
\begin{thm}\label{thm P^k planes BOM}\cite[Theorem 4.1]{BOZ}
    Let $E\subset \R^d$ be a Borel set supporting an irreducible $s$-Frostman probability measure $\mu$, then for all $1 \leq m \leq d-1$,
    \begin{equation}\label{BOZ k-affine planes}
 \hdim \mathcal P^{m}(E) \ge (m+1)\min\{s,d-m\}.
    \end{equation}
\end{thm}
As stated, this result is contingent on the existence of an irreducible $s$-$Frostman$ measure supported on $E$. However, there exists a simple argument that allows us to eliminate the irreducibility condition on $\mu$ while still obtaining the correct estimate of $ \hdim \mathcal P^{m}(E)$. Nevertheless, the penalty for dropping the irreducibility condition is that the estimate \eqref{BOZ k-affine planes} will no longer be true for all subplane dimensions $1\leq m\leq d-1$, but only on a restricted range of them, namely $1\leq m\leq d_{\mu}-1$, where $d_{\mu}\leq d$ is defined by \eqref{irreducible dimension of E}. This restricted version of Theorem \ref{thm P^k planes BOM} can be found in Lemma \ref{lemma:improvement k-planes}. When $d_\mu<d$, Lemma \ref{lemma:improvement k-planes} does not imply the continuum Beck-type Theorem that was the main object of study in \cite[Theorem 1.2]{BOZ}, but it is exactly the right tool to prove our main Theorem \ref{main} with no \textit{non-concentration} conditions imposed on $E$. The proofs for Lemma \ref{lemma:improvement k-planes} and Theorem \ref{main} can be found in \S \ref{Proof of main theorem}. 

\subsection{Non-concentration conditions: NC, Irreducibility and the Fubini Property} In the second part of the paper we present an alternative and more elementary approach towards obtaining an answer to the question addressed in Theorem \ref{thm P^k planes BOM}.

\begin{ques}\label{Question P^k planes}
    Let $E\subset \R^d$ be a Borel set supporting a $s$-$Frostman$ probability measure $\mu$ and $1 \leq k \leq d-1$. What \textbf{additional conditions} can be imposed on $E$ to obtain
    \[
    \hdim \mathcal P^{k}(E) \ge (k+1)\min\{s,d-k\}?
    \]
\end{ques}
As previously mentioned, Bright, Ortiz and Zakharov \cite{BOZ}, answered the Question \ref{Question P^k planes} under two different \emph{non-concentration} assumptions on the Borel set $E$.  The first assumption is about the set $E$ admitting an irreducible $s$-$Frostman$ measure in the sense of Definition \ref{defn: irreducibility of a measure}, this is just Theorem \ref{thm P^k planes BOM} restated here.
\begin{thm}\cite[Theorem 4.1]{BOZ}. Let $E\subset \R^d$ be a Borel set supporting an \textbf{irreducible} $s$-Frostman probability measure. Then for all $1 \leq k \leq d-1$,
    \[
    \hdim \mathcal P^{k}(E) \ge (k+1)\min\{s,d-k\}.
    \]
\end{thm}
Their main result, however, concerns the use of a second non-concentration notion on the Borel set $E$ in the case where no irreducible measure exists.
\begin{defn}\cite[Definition 1.1]{BOZ}. Let $d \ge 2$, and let $E\subset \R^d$ be Borel. We say that $E$ is \emph{non-concentrated}, or simply $\mathbf{NC}$, if for any $r \ge 1$ and any collection of affine subspaces $F_1,\dots, F_r\subset\R^d$ such that $\sum_{i=1}^r \hdim F_i \leq d-1$, we have 
    \[
    \hdim \Big(E\setminus \bigcup_{i=1}^r F_i\Big) = \hdim E.
    \]
\end{defn}
They realized that inside a set $E$ with the $\mathbf{NC}$ condition, it is possible to find $V_1,\dots, V_{d-1}$ affine flats and corresponding irreducible measures $\mu_1,\dots, \mu_{d-1}$ supported in $V_i\cap E$ in good position that together span sufficiently  thin and abundant hyperplanes.
\begin{thm}\label{thm:mainBOZ}\cite[Theorem 1.2]{BOZ}.
Let $E\subset \R^d$ be a Borel and $\mathbf{NC}$ set supporting a $s$-Frostman probability measure $\mu$.  Then, 
\[
\hdim \mathcal P^{d-1} (E) \geq d\min\{s,1\}.
\]
\end{thm}
Along the same lines, we also obtain an affirmative answer to the Question \ref{Question P^k planes}. However, we do this by using an approach based on the projection and slicing theorems by Marstrand and the use of a third notion of non-concentration distinct from the two considered previously; this non-concentration property is called the \emph{Fubini Property}, following the definition by H\'era, Keleti, and M\'ath\'e in \cite[Definition 2.5]{HKM}. 
\begin{defn}\label{Fubini definition}
 We say that a nonempty set $B\subset \R^m\times \R^n$  has the \textbf{Fubini Property}, if 
 \begin{equation}\label{Fubini property intro}
     \hdim(B)= \hdim(\pi_X(B))+\esssup^{\alpha}_{x\in \R^m}\hdim(\pi_X^{-1}(x)\cap B),
 \end{equation}
 where $\pi_X$ is the orthogonal projection onto $\R^m$. Here, the $\esssup$ is defined as follows: 
 
 Let $\alpha=\hdim(\pi_X(B))$ and $B_x=\pi_X^{-1}(x)\cap B$. If $\mathcal{H}^{\alpha}(\pi_X(B))>0$, then
\begin{align}
\label{eqessup}
\esssup_{x\in \R^m}^\alpha\hdim(B_x):= & \sup\{q \geq 0: \mathcal{H}^{\alpha}(\{x\in \R^m: \hdim(B_x) > q \}) >0\} \nonumber \\ 
=& \inf\{q \geq 0: \mathcal{H}^{\alpha}(\{x \in \R^m: \hdim(B_x)> q \}) =0\},
\end{align}
where $\sup(\emptyset)$ is taken to be zero. If $\mathcal{H}^{\alpha}(\pi_X(B))=0$, we let 
\begin{equation}
\label{essnull}
\esssup^\alpha_{x\in \R^m}\hdim(B_x)=\lim_{\varepsilon \to 0+} \esssup^{\alpha-\varepsilon}(\hdim(B_x)). 
\end{equation}
\end{defn}

The terminology makes reference to an analogy with the regular Fubini's theorem, meaning,  the dimension of $B$ is obtained by adding the dimension of the base set that parametrizes the vertical fibers plus the typical dimension of the fibers. Note that the following inequality always holds true
 \begin{equation}
     \displaystyle \hdim(\pi_X(B))+\esssup^\alpha_{x\in\R^m}\hdim(\pi_X^{-1}(x)\cap B)\leq \hdim(B),
 \end{equation}
 as can be seen by setting $\displaystyle t_{\varepsilon}=\esssup^\alpha_{x\in \R^m}\hdim(\pi_X^{-1}(x)\cap B)-\varepsilon$, with $\varepsilon>0$ and $\varepsilon\to0$ in equation \eqref{inequality essinf} (where $B_{x_0}=\pi_X^{-1}(x_0)\cap B$), see Figure \ref{fig:Fubini property} for an illustration.
 \begin{figure}[H]
    \centering
    \includegraphics[scale=0.18]{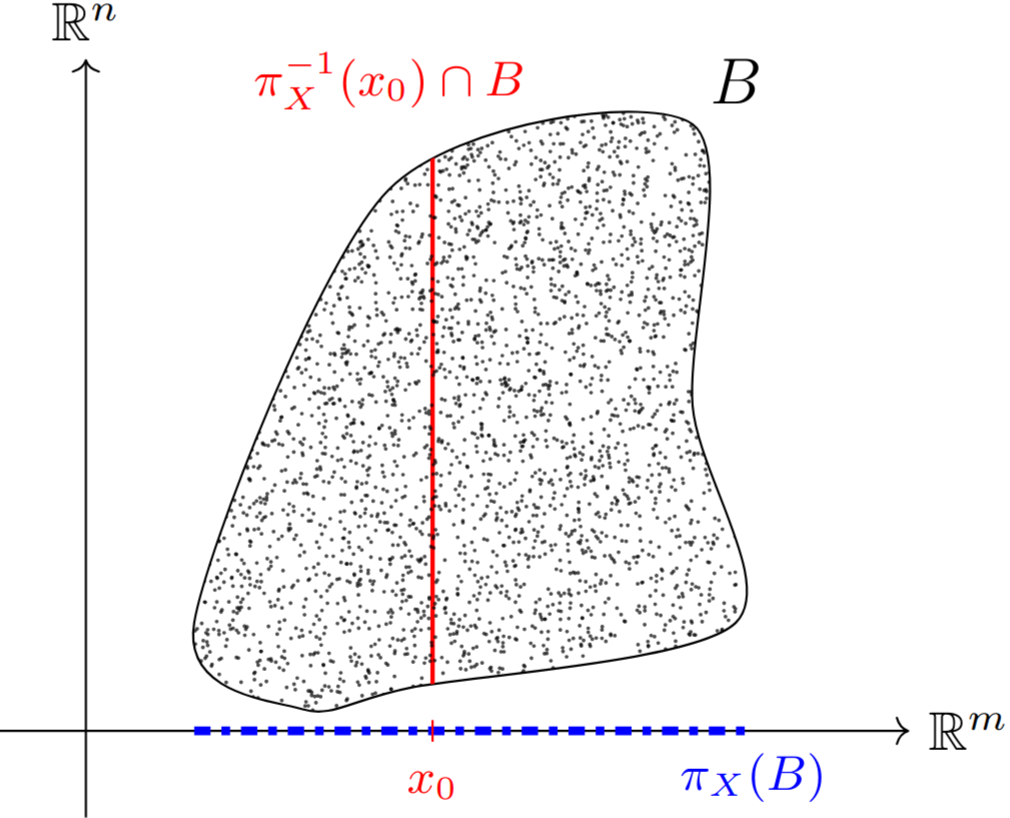}
\caption{Representation of $B\subset\R^{m+n}$ sliced by $\pi_X^{-1}$, in relation to the Fubini property.}
 \label{fig:Fubini property}
\end{figure}
 It is well known that the Fubini property \eqref{Fubini property intro} does not hold in general, for instance, Falconer in \cite{Falcbook2,Falcbook} provides the construction of two classic examples. See also \cite[Theorem 2.2]{HKM} and references therein for a related example with $f:[0,1]\to[0,1]^n$ whose graph has Hausdorff dimension $n+1$. 

\begin{example}\cite[Example 7.8]{Falcbook} There exist sets $E, F\subset\R$  with $\hdim E = \hdim F = 0$ and $\hdim(E \times F )=1$. In this case, given $B=E \times F$, then $\pi_X(B)=E$ and $\hdim(\pi_X^{-1}(x))=\hdim F=0$ for all $x\in \pi_X(B)$, therefore, the right-hand side of $\eqref{Fubini property intro}$ is $0+0$ while $\hdim(B)=1$.
\end{example}
Consider the following function 
\begin{align}\label{map Fubini}
\psi_k: &E^{k+1}\setminus D\subset \R^{(k+1)d} \xrightarrow{\hspace*{0.6cm}} \mathcal{A}(d,k) \\
&(x_0,x_1,\dots,x_k) \longmapsto \aff\{x_0,x_1,\dots,x_k\} \nonumber
\end{align}
That is, $\psi_k$ sends a $(k+1)$-tuple $\boldsymbol{x}=(x_0,x_1,\dots,x_k)$ to the affine span $\aff\{\boldsymbol{x}\}$. Note that $\psi_k$ is well defined only on the restriction of $E^{k+1}$ such that $\{x_0, \cdots, x_{k}\}$ are affine linearly independent (otherwise their span would be lower dimensional). We denote by $D$ the subset of degenerate spans. Thus, consider the following incidence set \begin{equation}\label{Incidence set S}
\mathcal{S}:=\{(\psi_k(\boldsymbol{x}),\boldsymbol{x}): \boldsymbol{x}\in E^{k+1}\setminus D\}\subset\A(d,k)\times\R^{d(k+1)}.    
\end{equation}
which is the inverted graph of $\psi_{k}$. This set can be embedded into a product of Euclidean spaces where the Fubini property \eqref{Fubini property intro} can be interpreted, we do this in \S \ref{Fubini section}, for the details, see Definition \ref{Fubini property in the affine grassmanian} and the discussion therein.

Also, let us use certain slice sections of $E^{k+1}\subset\R^{d(k+1)}$. For each $W\in \mathcal P^{k} (E)$, see that $W^{k+1}=W\times \dots \times W\subset\R^{d(k+1)}$ is a $k(k+1)$-affine plane in $\R^{d(k+1)}$ and thus define the following collection of diagonal affine subplanes restricted to $\mathcal P^{k} (E)$,
    \begin{equation}\label{Subset W^k}
\mathcal{W}^{k+1}(E):=\{W^{k+1}: W\in \mathcal P^{k} (E)\}\subset\A(d(k+1),k(k+1)).
\end{equation}
The collection $\mathcal{W}^{k+1}(E)$ is lower dimensional inside $\A(d(k+1),k(k+1))$. Let $\alpha$ denote the dimension of $\A(d(k+1),k(k+1))$, if we set $\beta=\hdim(\mathcal{W}^{k+1}(E))=\hdim( \mathcal P^{k} (E))$ then we can observe that $\beta\ll \alpha$. Thus, it is not guaranteed that in this smaller collection the dimension of the slices are not atypically large. Our answer to the Question \ref{Question P^k planes} is the following.
\begin{thm}\label{k-affine span Fubini}
Let $E\subset \R^d$ be a Borel set with $\hdim(E)>t>k$.  Let $F\subset E$ be a compact subset with $\hdim(F)\geq t$ satisfying the following
\begin{enumerate}
    \item The incidence set $\mathcal{S}$ of $F$ as in $\eqref{Incidence set S}$ satisfies the \textbf{Fubini property} from Definition \ref{Fubini property in the affine grassmanian}.

        \item The inequality 
        \begin{equation}\label{restricted W's assumption}
        \esssup^\beta_{W^{k+1}\in\mathcal{W}^{k+1}(F)} \hdim(F^{k+1}\cap W^{k+1})\leq\max\left\{0,\hdim(F^{k+1})-(k+1)(d-k)\right\} 
        \end{equation}
         holds true, where $\beta=\hdim(\mathcal P^{k} (F))$ and $\mathcal{W}^{k+1}(F)$  is given by $\eqref{Subset W^k}$. In other words, the slices of $F^{k+1}$ by elements in $\mathcal{W}^{k+1}(F)$ are not atypically large in the $\beta$-essential supremum sense.
\end{enumerate}
Then,
\begin{equation*}
\hdim \mathcal P^{k} (E) \geq \hdim \mathcal P^{k} (F) \geq (k+1)\min\{t,d-k\}. 
\end{equation*}
\end{thm}
For this result we imposed two conditions on the subset $F\subset E$, in \S \ref{Fubini section} we will analyze them in more detail and will also explain why they are natural assumptions to consider.

The non-concentration conditions considered in the three Theorems of this section require the set $E$ to be well spread along the $d$ dimensional space. These notions fail to capture the abundance of $\mathcal P^{k} (E)$ in the case when a large portion of $E$ is contained inside a lower dimensional subspace of $\R^d$. Let us take a look at the following example:

\begin{example}\label{exampple: failing 3 conditions}
Let $C\subset\R^1$ be the middle third Cantor set with dimension $c:=\hdim(C)=\dim_{\text{P}}(C)=\log_3(2)$. Let $E=[0,1]^2\times C\times \{0\}\subset \R^4$ and thus $2<\hdim(E)=2+c<3$. Since $\aff(E)\subset\R^3\times\{0\}$ it is easy to see that $\hdim(\mathcal P^{2}(E))$ is not close to $ \hdim(\A(4,2))=6$, but rather it satisfies $\hdim(\mathcal P^{2}(E))\leq\hdim(\A(3,2))=3$.

Indeed, we can verify that the dimensions are equal:  Let $a,b\in C$ with $a<b$, consider the horizontal sections $V_a=\R^2\times\{a\}\times\{0\}$ and $V_b=\R^2\times\{b\}\times\{0\}$ in $\A(4,2)$, note that $E\cap V_a$ and $E\cap V_b$ are both liftings of  $[0,1]^2$ by $a$ and $b$, respectively. Then, if $F=(E\cap V_a)\cup (E\cap V_b)$ it can be seen that $\mathcal{P}^2(F)$ contains a nonempty open subset of
$\mathcal A(3,2)$, in particular, $\hdim(\mathcal{P}^2(F))=\hdim(\A(3,2))$, the same is true for $\mathcal{P}^2(E)$. 
 
We will show that this set fails the three non-concentration conditions considered in this section, in particular, none of the theorems could be applied.
\begin{itemize}
    \item[(a)] Any $\mu$ $s-$Frostman measure supported on $E$ satisfies $\mu=\mu|_{\R^3\times\{0\}}$, then $\mu$ is not $\textbf{irreducible}$ on $\R^4$.
    \item[(b)] Similarly, given $F_1=\R^3\times\{0\}$, with $\hdim F_1=3$ and $\hdim (E\setminus F_1)=0<\hdim(E).$ Then $E$ is not \textbf{NC}.
    \item[(c)] By the above discussion we have that $\mathcal{P}^2(E)$ has positive natural invariant measure on the embedded copy of $\A(3,2)$ inside $\A(4,2)$, so by Marstrand's slicing theorem restricted to affine $2$-planes inside $\R^3\times \{0\}$ we get 
\begin{equation*}
 \hdim(E\cap W)=1+c   
\end{equation*}
for a positive measure collection of $W'$s contained in $\mathcal{P}^2(E)$, then we get
    \begin{align*}
    \esssup^3_{W\in\mathcal{P}^2(E)} \hdim(E^3\cap W^3)&\geq \esssup^3_{W\in\mathcal{P}^2(F)} \hdim(E^3\cap W^3)\\
    &= \esssup^3_{W\in\mathcal{P}^2(F)} \hdim\big((E\cap W)^3\big)\\
    &\geq 3(1+c), 
        \end{align*}
this number is larger than $3c$ which is the value given by $\eqref{restricted W's assumption}$ with $d=4$ and $k=2$, then $E$ fails the condition $\eqref{restricted W's assumption}$ as well.
\end{itemize}
\end{example}
As shown in Example \ref{exampple: failing 3 conditions}, certain sets may have plentiful affine spans, which may not be captured by the previous theorems. In order to address that, we introduce the following definition.

\begin{defn}
Given $E$ a Borel set with $k<\hdim(E)$. We define the \textbf{effective dimension} of the set $E$, to be the number 
\begin{equation}\label{effective dimension of E}
    d_E:=\min\{m \in \mathbb{N}\ : \exists \ V\in\mathcal{A}(d,m) \text{ such that }\hdim(E\setminus V)<\hdim(E) \}.
\end{equation}
Thus, $d_{E}$ represents the minimal affine dimension required to contain all of the maximally dimensional parts  of the set $E$ in $\mathbb{R}^d$. Note that by definition, $d_E$ exists, and it satisfies $k<\hdim(E)\leq d_E\leq d$.
\end{defn}
Notice that the set in Example \ref{exampple: failing 3 conditions} has $d_E=3$ with minimal subspace $V=\R^3\times\{0\}$, this shows why the dimension of  $\mathcal{P}^2(E)$ is not comparable to the dimension of $\mathcal{A}(4,2)$, but rather to that of  $\mathcal{A}(3,2)$. By using this notion of dimension we can capture abundance of affine $k$-planes, where the previous theorems in this section are unable to yield any results.
\begin{cor}\label{corollary: effective dimension}
    Let $E\subset \R^d$ be a Borel set with $\hdim(E)>t>k$, then $k+1\leq d_E\leq d$ and given $V\in\A(d,d_E)$ of minimal dimension satisfying $\eqref{effective dimension of E}$, let $F\subset E\cap V$ be a compact subset with $\hdim(F)\geq t$. After identifying $V$ with $\R^{d_E}$ assume that the following conditions hold:
\begin{enumerate}
    \item The incidence set $\mathcal{S}$ of $F$ as in $\eqref{Incidence set S}$ satisfies the \textbf{Fubini property} from Definition \ref{Fubini property in the affine grassmanian}.

        \item The inequality 
        \begin{equation}
        \esssup^\beta_{W^{k+1}\in\mathcal{W}^{k+1}(F)} \hdim(F^{k+1}\cap W^{k+1})\leq\max\left\{0,\hdim(F^{k+1})-(k+1)(d_E-k)\right\} 
        \end{equation}
         holds true, where $\beta=\hdim(\mathcal P^{k} (F))$ and $\mathcal{W}^{k+1}(F)$  is given as in $\eqref{Subset W^k}$.
\end{enumerate}
Then,
\begin{equation*}
\hdim \mathcal P^{k} (E)\geq \hdim \mathcal P^{k} (F) \geq (k+1)\min\{t,d_E-k\}. 
\end{equation*}
\end{cor}

\section{Preliminaries }
We will make use of the classic Marstrand theorems on projections and plane sections. These results were obtained by Marstrand \cite{Mar1} in the plane. Here we present the results in higher dimensions; for the projection theorems, see for example \cite[Theorem 6.1 and 6.2]{Falcbook}, and for the results on plane sections, see \cite[Theorem 10.10]{Mat1}.
\begin{thm}\label{mar1}
Let $E\subset\R^d$ be a Borel set. Then for $\gamma_{d,k}$ almost all $V\in \G(d,k)$,
\begin{equation}\label{mareq1}
\hdim (\Proj_V(E)) = \hdim E\ \text{if}\ \hdim E \leq k,
\end{equation}
and 
\begin{equation}\label{mareq12}
\mathcal L^{k}(\Proj_V(E))>0\ \text{if}\ \hdim E > k.
\end{equation}
\end{thm}
\begin{thm}\label{mar2}
Let $d-k < s \leq d$ and let $E\subset\R^d$ be $\mathcal H^s$ measurable with $0<\mathcal H^s(E)<\infty$. Then for  $\gamma_{d,k}$ almost all $H\in \G(d,k)$, 
\begin{equation}\label{mareq4}
 \mathcal{H}^{s-(d-k)} (E\cap (H+u))<\infty,
\end{equation}
for $\mathcal{L}^{d-k}$ almost all $u\in H^{\perp}$.
Further, for some positive $\mathcal{L}^{d-k}$ measure family of $u\in H^{\perp}$,

\begin{equation}\label{mareq3}
\hdim (E\cap (H+u)) = s-(d-k).
\end{equation}

\end{thm}
We used $d-k< s$ for convenience. Falconer proved the following sharp estimates for the set of exceptional projections in \cite{Falc1982}.

\begin{thm}\label{exmar2}
Let $d-k < s \leq d$ and let $E\subset\R^d$ be $\mathcal H^s$ measurable with $0<\mathcal H^s(E)<\infty$. Then there is a Borel set $\Theta:=\Theta(s,d,d-k)\subset \G(d,d-k)$ such that 
\begin{equation}\label{exsetbound}
\hdim \Theta\leq(d-k)-s+k(d-k)=(k+1)(d-k)-s
\end{equation}
and for all $V\in \G(d,d-k)\setminus \Theta,$
\begin{equation}\label{mareq5}
\mathcal{H}^{d-k}(\Proj_V(E))>0.
\end{equation}
\end{thm}
We state the following two forms of Hausdorff measure estimates on vertical and curved slices, \cite[Proposition 2.3]{Mat3} and the Coarea formula, see for example \cite[Theorem 3.10]{EG2015}.
\begin{prop}\label{union of vertical fibers}
Let $A\subset\mathbb{R}^{n+m}$, and set $A_x=\{y\in \mathbb{R}^n :(x,y)\in A\}$ for $x\in\mathbb{R}^m$. Then for any non-negative numbers $s$ and $t$
\begin{equation*}
    \int^*\mathcal{H}^t(A_x)\ d\mathcal{H}^s(x)\leq C(m,n,s,t)\mathcal{H}^{s+t}(A),\end{equation*}
where $\int^*$ denotes the upper integral. In particular, if 
\begin{equation}\label{inequality essinf}
\hdim(\{x\in\mathbb{R}^m: \hdim(A_x)\geq t\})\geq s, \ \text{ then }\ \hdim(A)\geq s+t.    
\end{equation}
\end{prop}
\begin{prop}[Coarea formula]\label{prop:coarea}
Let  $f:\mathbb{R}^d\to\mathbb{R}^m$ be Lipschitz continuous, $1\leq m\leq d$. For each $\mathcal{L}^d$- measurable set $A\subset\mathbb{R}^d$
\begin{equation}
    \int_A |J_m f(x)|\,d\mathcal{L}^d(x) =\int_{\mathbb{R}^m}\mathcal{H}^{d-m}\bigl(A\cap f^{-1}(y)\bigr)\,d\mathcal{L}^m(y),
\end{equation}
where $J_mf$ denotes the $m$-dimensional Jacobian of $f$.
\end{prop}
\section{Lemma \ref{lemma:improvement k-planes} and proof of Theorem \ref{main}}\label{Proof of main theorem}
We begin this section by presenting the observation that allows us to apply the result from Theorem \ref{thm P^k planes BOM} to obtain an unconditional version that holds true within a reduced range for the dimension of the $m$-planes considered in $\mathcal P^{m}(E)$.
\begin{defn}
Given a $s$-Frostman measure $\mu$ supported on $E$, we define the \textbf{irreducible dimension} of $\mu$, to be the number 
\begin{equation}\label{irreducible dimension of E}
    d_{\mu}:=\min\{m \in \mathbb{N}\ : \exists \ V\in\mathcal{A}(d,m) \text{ such that $\mu(V)>0$} \}.
\end{equation}
Note that $d_{\mu}=d$ if $\mu$ is irreducible in $\R^d$, otherwise $1\leq d_\mu\leq d-1$.
\end{defn}

\begin{lma}\label{lemma:improvement k-planes}
    Let $E\subset \R^d$ be a Borel set supporting an $s$-Frostman probability measure $\mu$ with $k<s$. Then $k+1\leq d_\mu\leq d$,  and for every $1 \leq m\leq d_\mu-1$,
    \begin{equation}
        \hdim \mathcal P^{m}(E) \ge (m+1)\min\{s,d_{\mu}-m\}
    \end{equation}
    regardless of the choice of the $s$-Frostman measure $\mu$ supported on $E$.
\end{lma}

\begin{proof}
If $\mu$ is irreducible in $\R^d$ as in Definition \ref{defn: irreducibility of a measure}, then $d_\mu=d$, and then Theorem \ref{thm P^k planes BOM} applies directly. If $\mu$ is not irreducible in $\R^d$, by definition there exists a proper affine plane, $V$, such that $\mu(E\cap V)>0$ and it can be chosen to be of minimal dimension, that is $V\in\mathcal{A}(d,d_{\mu})$. Observe that an $s$-Frostman measure assigns zero mass to every affine subspace of dimension strictly smaller than $s$, since $s>k$, we conclude $k+1\leq d_{\mu}\leq  d-1$. 

Let us fix one $V\in\mathcal{A}(d,d_{\mu})$ attaining the minimum in \eqref{irreducible dimension of E}, then the normalized restriction $\nu=\dfrac{\mu|_{V}}{\mu(V)}$ is a probability measure supported on $E':=E\cap V$ preserving the Frostman exponent of $\mu$, and by minimality of \eqref{irreducible dimension of E}, it is also irreducible in $V$.

After a translation and rotation we can regard $V=\R^{d_{\mu}}\times \{0\}^{d-d_\mu}$ and then it is possible to apply Theorem \ref{thm P^k planes BOM} with the $s$-Frostman irreducible measure $\nu$ on $E'\subset\R^{d_{\mu}}\times \{0\}^{d-d_\mu}$ where the ambient dimension is $d_{\mu}$. Since $\mathcal{P}^m(E')\subset\mathcal{P}^m(E)$, the lemma is proved.
\end{proof}The strategy of the proof of Theorem \ref{main}  can be summarized as follows. The $(k+1)$-simplex with base $\{x_0, \cdots, x_{k}\}\subset E$ and height $y\in E$, has volume given by
\begin{equation*}
 \Vol_{k+1}(x_0, \cdots, x_{k},y)=\frac{1}{k+1} \dist(y, W) \Vol_{k}(x_0, \cdots, x_{k})
\end{equation*}
where $W=\aff\{x_0, \cdots, x_{k}\}$ is the affine $k$-plane spanned by the collection of points forming the base, for this discussion assume that $W$ is indeed a linear subspace. The point $y$ provides the height and it comes from $y\in E\setminus W$, the corresponding height is $\dist(y, W)=|\Proj_{W^\perp}(y)|$.  If $E$ is not concentrated in few parallel affine planes $V=W+p$ with $p\in W^\perp$, then $\Proj_{W^\perp}(E)$ will have large measure and the same would be true for $\Vol_{k+1}^{(x_0,\ldots,x_{k})}(E)\sim|\Proj_{W^\perp}(E)|$. For this argument to work, we require to show that the Hausdorff dimension of the collection of all non-parallel affine planes that can hold a proper base of $k+1$ affinely independent points in $E$ surpasses that of the set of exceptional projections given in Theorem \ref{exmar2}.
\begin{figure}[H]
    \centering
    \includegraphics[scale=0.22]{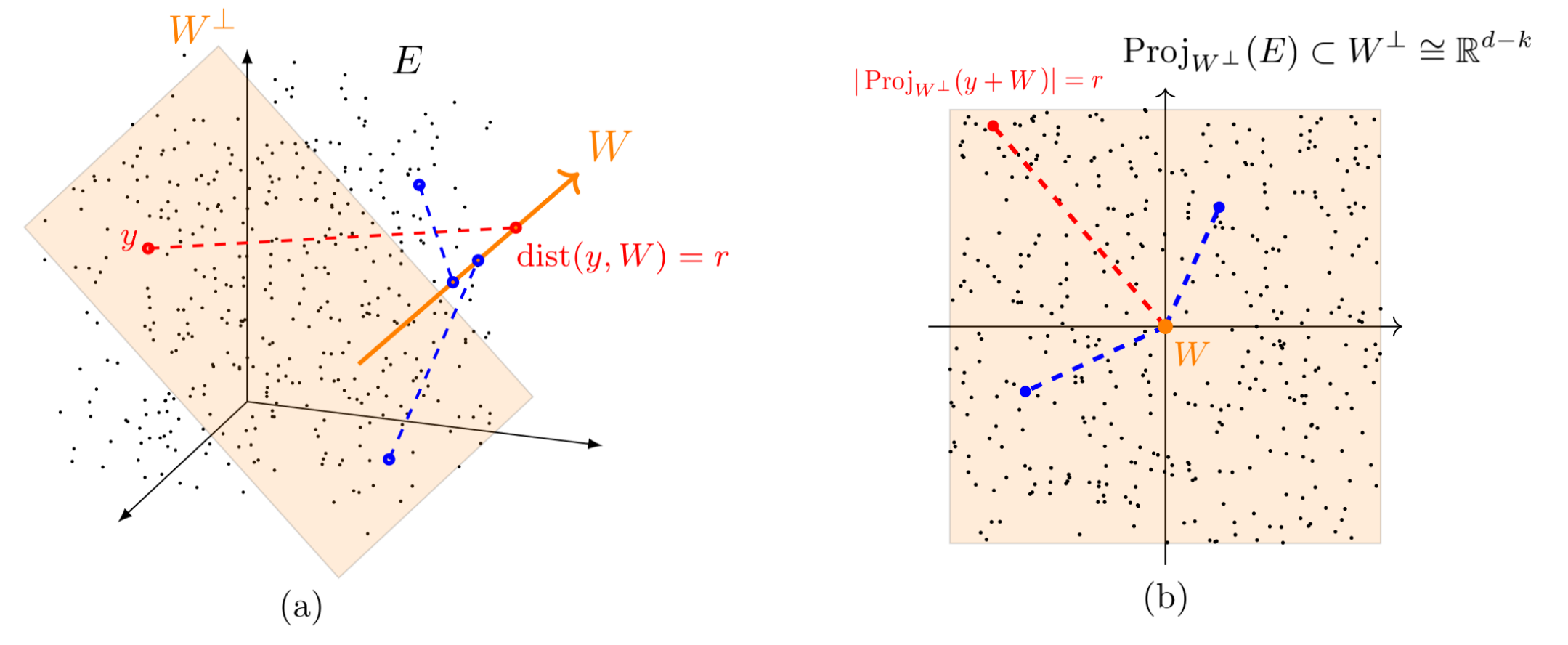}
    \caption{In Figure (a) we represent how the size of the set $|\Proj_{W^\perp}(E)|$ dictates the amount of different volumes of $k$-simplices with a base contained in $W$. In Figure (b) the origin is $\mathcal{O}=\Proj_{W^{\perp}}(W)$, and any  parallel affine plane  gets mapped to a point with distance $d(\mathcal{O},\Proj_{W^\perp}(y+W))=r$.}
    \label{fig:proof}
\end{figure}
\begin{proof}[Proof of Theorem \ref{main}]
We may assume without loss of generality that $E$ is a bounded Borel set such that $\hdim(E)>t>u>\max\{k,d-k\}$ and choose $\mu$ to be a $t$-$Frostman$ measure supported on $E$. By the definition \eqref{irreducible dimension of E}, there exists an affine plane $V\in \A(d,d_\mu)$ satisfying \eqref{irreducible dimension of E} with minimal dimension.  Set $E':=E\cap V$, then the restricted probability measure $$\nu=\dfrac{\mu|_{V}}{\mu(V)}$$ is $t$-Frostman supported on $E'$ and irreducible on $V$. We also have $\hdim(E')\geq t>u>d-k\geq d_{\mu}-k$, since $d_{\mu}\leq d$. We identify $V\cong \R^{d_\mu}$, and from now on we will regard $E'\subset\R^{d_\mu}$ as a Borel set with $\hdim(E')>u>\max\{k,d_\mu-k\}$. 

Let $W=\aff\{x_0, \cdots, x_{k}\}\in \A(d_{\mu},k)$ be the affine $k$-plane spanned by a collection of $(k+1)$-points $\{x_0, \cdots, x_{k}\}\subset E'$, this means that $W\in \mathcal P^{k}(E')$. In addition, we consider $p\in \R^{d_\mu}$ to be the closest point of $W\in \A(d_{\mu},k)$ to the origin, such that $W_0:=W-p\in \mathcal{G}(d_\mu,k)$. We will be interested in studying the set of $(k+1)$-volumes of simplices with base $W\in \mathcal P^{k}(E')$. As mentioned above, the idea is to 
observe that the $k$-simplex with base $\{x_0, \cdots, x_{k}\}$ and height $y$, has volume given by
\begin{equation}
 \Vol_{k+1}(x_0, \cdots, x_{k},y)=\frac{1}{k+1} \dist(y, W) \Vol_{k}(x_0, \cdots, x_{k}).
\end{equation}

The first step is to estimate the size of the family of directions of affine
$k$-planes spanned by $E'$, we need to verify that the collection 
\begin{equation}\label{PRojection onto the complement subset}
    \mathcal{P}_{\dir}^{\perp}:=\{W_0^{\perp}: W_0=W-p\in \G(d_\mu,k) \text{ for } W\in \mathcal{P}^k(E')\}\subset \G(d_\mu,d_\mu-k) ,
\end{equation}
is not fully contained in the set of exceptional projections, as defined in Theorem \ref{exmar2}. We first note that the subset from \eqref{PRojection onto the complement subset} has the same Hausdorff dimension as 
\begin{equation}
\mathcal{P}_{\dir}:=\{W_0:=W-p\in \G(d_\mu,k): W\in \mathcal{P}
^k(E')\},    
\end{equation} this follows by noticing that under a natural metric on the Grassmannian, for example, the one that embeds the metric space $\G(d_\mu,k)$ inside the real $d_\mu\times d_\mu$ square matrices, $M_{d_{\mu}}(\R)$ , by assigning $V\mapsto P_{V}$, where $P_V$ is the matrix projection onto $V$ with the basis independent metric given by
\begin{equation}\label{metric on the Grassmanian}
d(W,V)=||P_W-P_V||_{2},    
\end{equation}
where $|| \cdot||_{2}$ is the Frobenius norm in $M_{d_{\mu}}(\R)$; we observe that $||P_W-P_V||_{2}=||P_{W^\perp}-P_{V^\perp}||_{2}=d(W^\perp,V^\perp)$,  meaning that $\mathcal{P}_{dir}$ and $\mathcal{P}_{dir}^{\perp}$ are isometric, thus, their Hausdorff dimension agrees. 

We now use the local product structure of the affine Grassmannian to estimate $\hdim\left(\mathcal{P}_{dir}\right),$ the direction map $$\operatorname{dir}: \mathcal A(d_{\mu},k)\longrightarrow \G(d_{\mu},k)$$ given by $\dir(W_0+p)=W_0$ where $p\in W_0^{\perp}$ is the closest point of $W\in \A(d_{\mu},k)$ to the origin. The direction map makes $\A(d_{\mu},k)$ into a smooth vector bundle over $\G(d_{\mu},k)$, whose fiber above $W_0\in \G(d_\mu,k)$ is naturally identified with $W_0^\perp \simeq \R^{d_\mu-k}.$  Moreover, we get 
\begin{equation*}
    \dir (\mathcal{P}^k(E'))=\mathcal{P}_{dir},
\end{equation*}
then $\mathcal{P}^k(E')$ is locally parametrized by a subset of $\mathcal{P}_{\dir}\times \R^{d_\mu-k}$. Choose a countable collection of relatively compact coordinate neighborhoods $U_j\subset \G(d_{\mu},k)$ covering $\mathcal{P}_{\dir}$, with local trivialization $(U_j,\varphi_j)$. Write $$\mathcal P_j:=\mathcal P^k(E')\cap \operatorname{dir}^{-1}(U_j).$$ Since $E$ is bounded, $E'$ is as well,  then there exists $R$ such that $E'\subset[-R,R]^{d_{\mu}}$ then for each local trivialization, we have $\varphi_j(\mathcal P_j)\subset (\mathcal{P}_{\dir}\cap U_j)\times [-C_jR,C_jR]^{d_{\mu}-k}$, for some $0<C_j<\infty.$ On bounded pieces, the corresponding coordinate maps and their inverses are Lipschitz and the product with the Euclidean fiber has the
expected dimension. Then
\begin{align}
\hdim(\mathcal P_j)=\hdim\bigl(\varphi_j(\mathcal P_j)\bigr)&\leq\hdim\bigl((\mathcal{P}_{\dir}\cap U_j)\times [-R,R]^{d_{\mu}-k}\bigr)\\ \nonumber
&=\hdim(\mathcal{P}_{\dir}\cap U_j)+(d_\mu-k)\\   &\leq \hdim(\mathcal{P}_{\dir})+ (d_\mu-k) \nonumber
\end{align}
By $\displaystyle \mathcal{P}^k(E')=\bigcup_{j}\mathcal{P}_j$ and the fact that Hausdorff dimension is countably stable, we conclude 
\begin{equation}\label{P_dir}
 \hdim(\mathcal{P}^{k}(E'))\leq \hdim(\mathcal P_{dir})+(d_\mu-k)  
\end{equation}
Now we apply Lemma \ref{lemma:improvement k-planes} to the irreducible $t$-Frostman measure $\nu$ and since $t>d_\mu-k$, we get $\hdim(\mathcal{P}
^k(E'))\geq (k+1)(d_\mu-k)$,  we also know that $\hdim(\mathcal P_{dir}) =\hdim(\mathcal P_{dir}^{\perp})$ so \eqref{P_dir} can be finally rewritten as
\begin{equation}\label{P_dirperp}
\hdim(\mathcal P_{dir}^{\perp}) \geq k(d_\mu-k).   
\end{equation}
Since $\mathcal{P}_{dir}^\perp \subseteq \G(d_{\mu},d_{\mu}-k)$ and  $\dim\left(\G(d_{\mu},d_{\mu}-k)\right)=k(d_\mu-k)$, indeed \eqref{P_dirperp} is an equality throughout. Next we apply the exceptional projection theorem, since  $\hdim(E')>u> d_\mu-k$, then there exists a compact subset $F\subset E'$ such that  $0<\mathcal H^u(F)<\infty$, in this way we can apply Theorem \ref{exmar2} to $F\subset\R^{d_{\mu}}$ with exponent $u$ and projection dimension $d_{\mu}-k$. By the estimates given by Theorem \ref{exmar2}, the exceptional set of projections $\Theta:=\Theta(u,d_\mu,d_\mu-k)$ has $\hdim\Theta\leq(k+1)(d_\mu-k)-u$, combining this with \eqref{P_dirperp}  we obtain the dimensional threshold needed for $\hdim\left(\mathcal P_{dir}^{\perp}\right)>\hdim \Theta$, it is 
\begin{equation}\label{ineq s}
   k(d_\mu-k)>(k+1)(d_\mu-k)-u 
\end{equation}
or $d_\mu-k< u$, which has already been assumed. This proves the existence of at least one $W^{\perp}_0\in\mathcal P_{dir}^{\perp}\subset\G(d_\mu,d_\mu-k)$ that lies outside of the set of exceptional projections, $\Theta$. Therefore, \begin{equation}\label{positive d-k Lebesgue measure}
\mathcal{L}^{d_\mu-k}(\Proj_{W_0^{\perp}}(F-p))>0,
\end{equation} 
where $p$ is chosen so that $W:=W_0+p\in \mathcal{P}^k(E')$, choose affinely independent $x_0,\dots, x_k\in E'$ such that $W=\aff\{ x_0,\dots, x_k\}$. 

The last step is to use \eqref{positive d-k Lebesgue measure} to show that the set of radii $\mathcal{R}:=\{|\Proj_{W_0^{\perp}}(y-p)|: y\in F \}\subset \R$ has positive one dimensional Lebesgue measure, see Figure \ref{fig:proof} (b). This can be achieved by using the Coarea formula given in Proposition \ref{prop:coarea}, consider the radial map $r:\R^{d_{\mu}-k}\longrightarrow\R$ defined by $r(z)=|z|$ which is Lipschitz, and $|\nabla r(z)|=J_1r(z)=1$ for all $z\neq 0$, then we get 
\begin{equation}
    0<\int_{\Proj_{W_0^{\perp}}(F-p)} d\mathcal{L}^{d_{\mu}-k}(z) =\int_{0}^{\infty}\mathcal{H}^{{d_{\mu}-k}-1}\bigl(\Proj_{W_0^{\perp}}(F-p)\cap r^{-1}(t)\bigr)\,dt,
\end{equation}
where the left hand side is \eqref{positive d-k Lebesgue measure} and  $r^{-1}(t)=\mathbb{S}^{d_{\mu}-k-1}(t)$ is the sphere on $\R^{d_{\mu}-k}$ of radius $t$. Since the right hand side integral is positive, then $\mathcal{H}^{{d_{\mu}-k}-1}\bigl(\Proj_{W_0^{\perp}}(F-p)\cap r^{-1}(t)\bigr)>0$ for some subset of $t'$s of positive $\mathcal{L}^1$ measure, therefore, the set of radii, $\mathcal{R}$, has positive  $\mathcal{L}^1$ measure, this is $\mathcal{L}^1(\{|\Proj_{W_0^{ \perp}}(y-p)|: y\in F\} )>0$. Now we are ready to conclude the theorem, since
\[
  \Vol_{k+1}^{(x_0, \cdots, x_{k})}(F)=\frac{1}{k+1}\Vol_{k}(x_0, \cdots, x_{k}) \{\dist(y, W):y\in F\}  ,
\] 
combined with the fact that $\{\dist(y,W): y\in F\}=\{|\Proj_{W_0^\perp}(y-p)|: y\in F\}$ shows that $\mathcal{L}^1\left(\Vol_{k+1}^{(x_0,\ldots,x_{k})}(F)\right)>0$  and the same conclusion holds for $E\supset E'\supset F$, finishing the proof.
\end{proof}
\section{Fubini Property and Proof of Theorem \ref{k-affine span Fubini}}\label{Fubini section}
In this section we provide some context to the Fubini property and explain why it is a natural non-concentration assumption which in the context of Theorem \ref{k-affine span Fubini} is not particularly restrictive. We recall Definition \ref{Fubini definition} from H\'era, Keleti, and M\'ath\'e \cite{HKM}.
\begin{defn}\label{Fubini property main}
We say that a nonempty set $B\subset \R^m\times \R^n$  has the \textbf{Fubini Property}, if 
 \begin{equation}\label{Fubini property}
     \hdim(B)= \hdim(\pi_X(B))+\esssup^{\alpha}_{x\in \R^m}\hdim(\pi_X^{-1}(x)\cap B),
 \end{equation}
 where $\pi_X$ is the orthogonal projection onto $\R^m$ and $\alpha=\hdim(\pi_X(B))$.
\end{defn}
It was already mentioned in the introduction that the following inequality always holds true
 \begin{equation}\label{Fubini always true}
     \hdim(\pi_X(B))+\esssup^{\alpha}_{x\in \R^m}\hdim(\pi_X^{-1}(x)\cap B)\leq \hdim(B).
 \end{equation}
A classic example of a set $B\subset\R^{m+n}$ failing the Fubini property is given in the case where $B$ is the graph of a fractal space-filling curve, see for example, Figure \ref{fig:graph}.
 
\begin{example}\cite[Theorem 8.2]{Falcbook2}\label{Example graph} For every $1<s<2$, there exists a function $f_s:[0,1]\to \R$ whose graph $B=\{(x,f_s(x)):x\in[0,1]\}\subset\R^{1+1}$ satisfies $\hdim(B)=s$. In this case, $\pi_X(B)=[0,1]$ and $\pi_X^{-1}(x)\cap B$  is the singleton $(x,f_s(x))$ for all $x\in [0,1]$, therefore, the right-hand side of $\eqref{Fubini property}$ is $1+0$ while $\hdim(B)=s>1$.
\end{example}
\begin{figure}[H]
    \centering
    \includegraphics[scale=0.2]{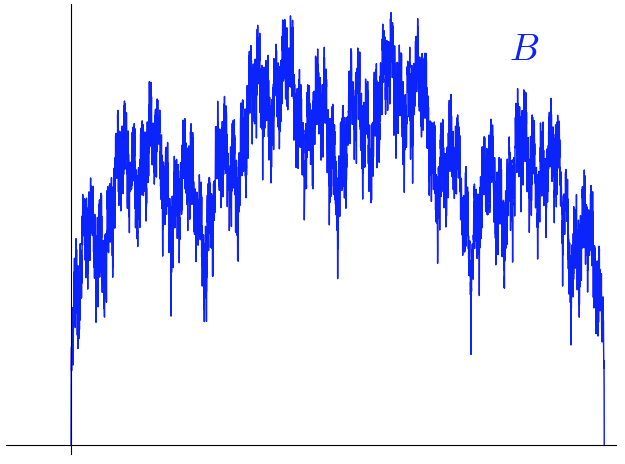}
    \caption{Graph of some $f_s$ with Hausdorff dimension $1<s<2$.}
    \label{fig:graph}
\end{figure}
The set $B$ from Figure \ref{fig:graph} does not satisfy the Fubini property along the vertical direction, but what about other directions? For example, the graph from Figure \ref{fig:graph} may satisfy a Fubini property along horizontal slices $\pi_Y^{-1}$; it would do it if $\displaystyle \esssup^{1}_{y_0\in \{0\}\times \R}\hdim(\pi_Y^{-1}(y_0)\cap B)=s-1$.  

Remember the function 
\begin{align}\label{map Fubini section 4}
\psi_k: &E^{k+1}\setminus D\subset \R^{(k+1)d} \xrightarrow{\hspace*{0.6cm}} \mathcal{A}(d,k) \\
&(x_0,x_1,\dots,x_k) \longmapsto \aff\{x_0,x_1,\dots,x_k\} \nonumber
\end{align}
where $D$  denotes the subset of tuples with degenerate affine spans. We are interested on the following incidence set \begin{equation}\label{Incidence set S sect 4}
\mathcal{S}:=\{(\psi_k(\boldsymbol{x}),\boldsymbol{x}): \boldsymbol{x}\in E^{k+1}\setminus D\}\subset\A(d,k)\times\R^{d(k+1)}.    
\end{equation}
We would like to define a notion of Fubini property in sets of the form \eqref{Incidence set S sect 4}, we achieve this by first embedding $\mathcal{S}$ into a product of Euclidean spaces. Consider the Euclidean embedding
\begin{equation}\label{embedding affine grassmanian}
\iota: \A(d,k)\hookrightarrow M_d(\R)\times \R^d, \qquad \iota(W_0+p)=(P_{W_0},p),  
\end{equation}
where $W=W_0+p$ with $p\in W_{0}^{\perp}$ and $P_{W_0}$ denotes the orthogonal projection matrix onto $W_0$. Also note that $\iota$ is bi-Lipschitz for the affine Grassmannian metric defined in \eqref{metric on the affine Grassmanian}.

\begin{defn}\label{Fubini property in the affine grassmanian}
We say that the incidence set $\mathcal{S}\subset \A(d,k)\times \R^n$ of $E$ defined by \eqref{Incidence set S sect 4}  has the \textbf{Fubini Property}, if the embedded version
\begin{equation}\label{fubini equation affine grassmanian}
\big\{\big( \iota(\psi_k(\boldsymbol{x})),\boldsymbol{x}  \big):\boldsymbol{x}\in E^{k+1}\setminus D \big\}
\end{equation}
satisfies the regular \textbf{Fubini Property} from Definition \ref{Fubini property main}.
\end{defn}

An application of the Marstrand's Projection Theorem in higher dimensions proves that any nonempty Borel set satisfies the Fubini property along almost every subspace direction.
 \begin{prop}\cite[Proposition 2.6]{HKM}
Any nonempty Borel set $B\subset\R^{m+n}$ has the
Fubini property along almost every $m$-dimensional
subspace; that is,
 \begin{equation}\label{Fubini property along V}
     \hdim(B)= \hdim(\pi_V(B))+\esssup^{\alpha}_{x\in V}\hdim(\pi_V^{-1}(x)\cap B),
 \end{equation}
holds for $\gamma_{m+n,m}$-almost every $V \in \G(m+n,m)$, where $\alpha=\hdim(\pi_V(B))$ and $\pi_V:=\Proj_V$ is the orthogonal projection onto $V$ with $\gamma_{m+n,m}$ denoting the Haar measure on the Grassmannian. 
 \end{prop}
 \begin{figure}[H]
    \centering
    \includegraphics[scale=0.27]{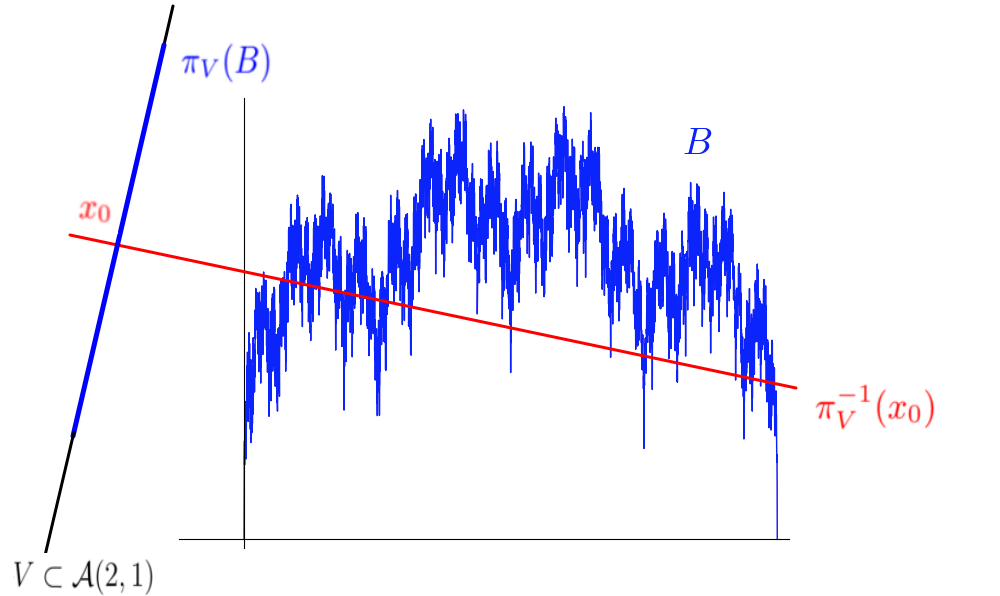}
    \caption{Fubini property along different directions.}
    \label{fig:graph 2}
\end{figure}

Notice that the first condition of Theorem \ref{k-affine span Fubini}, that demands the incidence set $\eqref{Incidence set S}$ to satisfy the Fubini property is equivalent to requiring the graph of $\psi_{k}$ to have the Fubini property along the horizontal direction $\pi_Y:=\Proj_{\{0\}^m\times\R^n}$, which is a plausible situation, at least for a graph of a function $f_s$ as the one from Figure \ref{fig:graph 2}. The situation also seems plausible for the graph of $\psi_{k}$ if it is the case that the fibers $\psi_{k}^{-1}(W)$ are large for almost every $W\in \mathcal{P}^k(E)$, namely, if there are many collections $\{x_0,\dots,x_{k}\}\subset(E\cap W)$ that span $W=\aff\{x_0,\dots,x_{k}\}$ and typically it happens that $\hdim(E\cap W)= \hdim(E)-(d-k)>0$. The next lemma shows why we imposed the Fubini property condition on the incidence set $\mathcal{S}$ rather than on the set $E$ itself.
\begin{lma}\label{fibers}
    Let $\psi: \mathcal{D}\subset\R^m \to \R^n$ be a Lipschitz map into $\R^n$. Assume that the graph $$\Gamma:=\{(x,\psi(x)):x\in \mathcal{D}\}\subset \R^{m}\times\R^n$$
    satisfies the Fubini property along the horizontal direction  $\pi_Y$, as in $\eqref{Fubini property along V}$,  this is
    \begin{equation}\label{Lipschits graph 1}
  \hdim \Gamma=\hdim \pi_Y(\Gamma)+ \esssup^\alpha_{y\in \R^n}\hdim(\pi_Y^{-1}(y)\cap \Gamma),
    \end{equation}
 where $\alpha=\hdim(\pi_Y(\Gamma))$.  Then, the following Fubini-type property holds true
\begin{equation}\label{Lipschits graph 2}
  \hdim (\mathcal{D})=\hdim \psi(\mathcal{D})+ \esssup^\alpha_{y\in \R^n}\hdim(\psi^{-1}(y)\cap \mathcal{D}).
    \end{equation}   
\end{lma}
\begin{proof}
Define the graph map $\Psi:\mathcal{D}\to \Gamma$ by $\Psi(x):=(x,\psi(x))$. Since $\psi$ is Lipschitz on its domain, then by the triangle inequality we see that for any $x,x'\in \mathcal{D}$,
\begin{equation*}
    |\Psi(x)-\Psi(x')|\leq |x-x'|+|\psi(x)-\psi(x')|\leq |x-x'|+C|x-x'|=(C+1)|x-x'|,
\end{equation*} 
which shows that $\Psi$ is Lipschitz as well. Moreover, $\Psi$ is invertible with inverse $\Psi^{-1}=\pi_X|_{\Psi(\mathcal{D})}$ which is also Lipschitz because it is an orthogonal projection. Then, the map $\Psi$ is bi-Lipschitz, and thus its action preserves the Hausdorff dimension of the domain $$\hdim(\mathcal{D})=\hdim(\Psi(\mathcal{D}))=\hdim(\Gamma),$$ where these are the left-hand side terms from the statement, i.e., \eqref{Lipschits graph 1} and \eqref{Lipschits graph 2}. The two terms in the right-hand side follow directly by the definitions of $\pi_Y$ and $\psi$. We observe that $\pi_Y(\Gamma)=\psi(\mathcal{D})$ and for each fiber, fix $y\in \psi(\mathcal{D})\subset\R^n$ to get
\begin{equation*}
    \pi_Y^{-1}(y)\cap\Gamma=\{(x,y)\in \Gamma:x\in \psi^{-1}(y)\},
\end{equation*}
 the above is just a lifting by $y$ of the set $\psi^{-1}(y)\cap \mathcal{D}=\{x\in \mathcal{D}:\psi(x)=y\}$, thus, the Hausdorff dimension is preserved as well, meaning that $\hdim\bigl(\pi_Y^{-1}(y)\cap \Gamma\bigr)=\hdim (\psi^{-1}(y)\cap \mathcal{D})$ for each $y\in\psi(\mathcal{D})$. In this way we see that the two terms on the right of \eqref{Lipschits graph 1} and \eqref{Lipschits graph 2} also match.
\end{proof}
\begin{lma}\label{D is lower dimensional}
Let $E\subset \R^d$ be bounded with $\hdim(E)>k$. The subset of degenerate spans $D:=\{\boldsymbol{y}\in E^{k+1}: \dim(\aff\{\boldsymbol{y}\})<k\}$ is a lower dimensional subset of $E^{k+1}$, and more precisely, every degenerate tuple, $\boldsymbol{y}\in D$, has one point of the form  \begin{equation*}
    y_i=\sum_{j\neq i}\lambda_jy_j
\end{equation*}
with coefficients belonging to $\displaystyle
    \Lambda_{i}=\{(\lambda_j)_{j\neq i}\in[-1,1]^k\ : \sum_{j\neq i}\lambda_j=1 \ \}$. Consequently, we have 
    \begin{equation}
        \hdim(E^{k+1} \setminus D )=\hdim(E^{k+1}).
    \end{equation}
\end{lma}
\begin{proof}
For any $\boldsymbol{y}\in D$, by affine dependency there are coefficients $c_1,\dots,c_k$, not all zero, such that $\displaystyle \sum_{j=1}^{k}c_j(y_j-y_0)=0$, by expanding the sums we get $\displaystyle \sum_{j=1}^{k}c_jy_j-y_0\sum_{j=1}^{k}c_j=0$, define $\displaystyle c_0:=-\sum_{j=1}^kc_j$ and we select the value $i\in\{0,\dots,k\}$ such that $\displaystyle|c_i|=\max_{0\leq j\leq k}\{|c_j|\}$. In particular, $c_i\neq0$, and we obtain 
\begin{equation}
    y_i=\sum_{j\neq i}\lambda_jy_j
\end{equation}
where $\lambda_j=-\dfrac{c_j}{c_i}$, so $|\lambda_j|\leq1$, and by the definition of $c_0$ we can see that $\sum_{j\neq i}\lambda_j=1$. Thus, every degenerate tuple has one point which is an affine combination of the other $k$ points with coefficients bounded by $1$. Let us define \begin{equation}
    \Lambda_{i}=\{(\lambda_j)_{j\neq i}\in[-1,1]^k\ : \sum_{j\neq i}\lambda_j=1 \ \},
\end{equation}
this is the proof of the first part. The collection of coefficients $\Lambda_i$ is contained in some $(k-1)$-dimensional affine hyperplane intersecting $[-1,1]^k$. For each $i\in\{0,\dots,k\}$, define the map $
\Phi_i:E^k\times\Lambda_i\longrightarrow (\R^d)^{k+1}$ by inserting 
$y_i= \sum_{j\neq i}\lambda_jy_j$ in the $i$-th position, this is
\begin{equation}\label{map phi_i}
\displaystyle\Phi_i\left( (y_j)_{j\neq i} ,\boldsymbol{\lambda}\right)=(y_0,\dots,y_{i-1},\ \sum_{j\neq i}\lambda_jy_j\ ,y_{i+1},\dots,y_k),
\end{equation}
this removes the dependency on the point $y_i$. By the discussion above, every $\boldsymbol{y}\in D$ has a representation as in \eqref{map phi_i} for some $\boldsymbol{\lambda}\in\bigcup_{j}\Lambda_j$. Therefore, we have the relation 
\begin{equation}
 D\subseteq\bigcup_{i=0}^k\Phi_i(E^k\times\Lambda_i),   
\end{equation}
the boundedness of $E$ and $\Lambda_i$ makes each of the $\Phi_i$ a Lipschitz continuous function, then by countable stability of Hausdorff dimension,  we conclude
\begin{align*}
    \hdim(D)&\leq\max_{0\leq i\leq k}\{\hdim(\Phi_i(E^k\times\Lambda_i))\}\leq\hdim\big(E^k\times\R^{k-1}\big),
\end{align*}
the last inequality is justified by the fact that each $\Lambda_i$ is contained in a $(k-1)$-dimensional affine hyperplane intersecting $[-1,1]^k$, then by the standard lower Hausdorff dimension inequality for Cartesian products
\begin{equation*}
\hdim\big(E^k\times\R^{k-1}\big)=\hdim(E^k)+(k-1)<\hdim(E^k)+\hdim(E)\leq\hdim(E^{k+1}),
\end{equation*}
we conclude $\hdim(D)\leq\hdim(E^k\times\R^{k-1})$.
\end{proof}

Regarding the second assumption from Theorem \ref{k-affine span Fubini}, remember that for each $W\in \mathcal P^{k} (E)$, we considered $W^{k+1}=W\times \dots \times W\subset\R^{d(k+1)}$ which is a $k(k+1)$-affine plane in $\R^{d(k+1)}$. We define the following collection
\begin{equation}\label{Subset W^k 2}
\mathcal{W}^{k+1}(E):=\{W^{k+1}: W\in \mathcal P^{k} (E)\}\subset\A(d(k+1),k(k+1)).
\end{equation}
In $\A(d,k)$ we consider the natural invariant product measure induced on the affine Grassmannian by 
\begin{equation}
d\sigma_{d,k}(W_0+p)=d\gamma_{d,k}(W_0) d \mathcal{H}^{d-k}_{W_0^\perp}(p).    
\end{equation}
Observe that for generic slices of $E^{k+1}$, Marstrand's Theorem \ref{mar2} would give the expected 
\begin{equation*}
 \hdim(E^{k+1}\cap V)\leq \max\left\{0,\hdim(E^{k+1})-(k+1)(d-k)\right\},  
\end{equation*}
for $\sigma_{d(k+1),k(k+1)}$ almost every $V\in\A(d(k+1),k(k+1))$, and the maximum depends entirely on the constant $\hdim(E^{k+1})$. Hence, we conclude that 
\begin{equation}\label{Slicing E^{k+1} 2}
    \esssup^\alpha_{V\in\A(d(k+1),k(k+1))} \hdim(E^{k+1}\cap V) = \max\left\{0,\hdim(E^{k+1})-(k+1)(d-k)\right\},
\end{equation}
where $\alpha=\hdim(\A(d(k+1),k(k+1)))$. Nevertheless, the difficulty lies in the fact that \eqref{Slicing E^{k+1} 2} gives no information about the intersection with the restricted family of diagonal affine planes from \eqref{Subset W^k 2}, since $\mathcal{W}^{k+1}(E)$ is a $\sigma_{d(k+1),k(k+1)}$ null set, since $\beta=\hdim(\mathcal{W}^{k+1}(E))<\alpha$.
This is the reason why, the value 
\begin{equation}
\esssup^\beta_{W^{k+1}\in\mathcal{W}^{k+1}(E)} \hdim(E^{k+1}\cap W^{k+1}),
\end{equation}
is part of the second assumption  of Theorem \ref{k-affine span Fubini}. In other words, we ask that the sections $E^{k+1}\cap W^{k+1}$ are not exceptionally large in comparison to the generic size in the sense of the natural invariant measure on $\A(d(k+1),k(k+1))$. 
After this discussion on the two assumptions we can begin the proof of the Theorem.
\begin{proof}[Proof of Theorem \ref{k-affine span Fubini}]
We will work with the compact $F\subset E$ considered in the statement of the theorem. Remember the function \begin{align*}
\psi_k: &F^{k+1} \setminus D \subset \R^{(k+1)d} \xrightarrow{\hspace*{0.6cm}} \mathcal{A}(d,k) \\
&(x_0,\dots,x_k) \longmapsto \aff\{x_0,\dots,x_k\}, 
\end{align*}
that sends a $(k+1)$-tuple $\boldsymbol{x}=(x_0,\dots,x_k)$ to their affine span $\aff\{\boldsymbol{x}\}$ and $D$ is the subset of degenerate tuples. We have seen in Lemma \ref{D is lower dimensional} that $D$ is a lower dimensional subset of $F^{k+1}$ and therefore $$\hdim(F^{k+1}\setminus D)=\hdim(F^{k+1}).$$ Note that $\psi_k$ is a locally Lipschitz map away from $D$  under the natural metric on $\A(d,k)$ given by
\begin{equation}\label{metric on the affine Grassmanian}
d(W,V)=||P_{W_0}-P_{V_0}||_{2}+|p-q|,    
\end{equation}
where we used the unique decomposition $W=W_0+p\in \A(d,k)$ with $W_0\in \G(d,k)$ and $p\in W_0^\perp$, similarly for $V=V_0+q\in \A(d,k)$.

Now we will see that the result of Lemma \ref{fibers} can still be used for $\psi_k$ even though it is not a globally Lipschitz map. For each $\boldsymbol{x}\in F^{k+1}$, assign the $d\times k$ matrix 
\begin{equation}\label{matrix column vectors for degeneracy}
\text{M}(\boldsymbol{x})=(x_1-x_0\ \cdots \ x_k-x_0),    
\end{equation}
whose $i$-th columns is the column vector $x_i-x_0$, so we have an equivalent definition 
\begin{equation*}
D=\{\boldsymbol{x}\in F^{k+1}: \rank(\M(\boldsymbol{x}))<k\}.    
\end{equation*}
However, it could be the case that although $\M(\boldsymbol{x})$ has full rank, the inverse of $\M(\boldsymbol{x})^T\M(\boldsymbol{x})$ might not be uniformly bounded as $\boldsymbol{x}$ approaches $D$. Then, a more quantitative separation condition is given by 
\begin{equation}
    X_n=\big\{\boldsymbol{x}\in F^{k+1}: \ \det\left(\M(\boldsymbol{x})^T\M(\boldsymbol{x}) \right)>\frac{1}{n} \big\}\subset F^{k+1}\setminus D,
\end{equation}
and thus we obtain $F^{k+1}\setminus D=\bigcup_{n=1}^{\infty}X_n$. Consider the incidence set  \begin{equation}\label{Incidence set S 2}
\mathcal{S}:=\{\left(\psi_k(\boldsymbol{x}),\boldsymbol{x}\right): \boldsymbol{x}\in F^{k+1}\setminus D\},    
\end{equation}
and similarly 
\begin{equation}
\mathcal{S}_n:=\{\left(\psi_k(\boldsymbol{x}),\boldsymbol{x}\right): \boldsymbol{x}\in X_n\},  \end{equation}

Now, by the uniform separation of each of the $X_n$ away from $D$ we have that $\psi_k|_{X_n}$ is Lipschitz on each $X_n$, meaning that the inverted graph $\mathcal{S}_n$ preserves the Hausdorff dimension, then $\hdim(\mathcal{S}_n)=\hdim(X_n)$. Thus, we get
\begin{align}\label{countable union for S}
\hdim(\mathcal{S})&=\sup_{n}\{\hdim(\mathcal{S}_n)\}\\ \nonumber
    &=\sup_{n}\{\hdim(X_n)\}\\ \nonumber
    &=\hdim(F^{k+1}\setminus D)\\ \nonumber
    &=\hdim(F^{k+1}). \nonumber
\end{align}

It is assumed that $\mathcal{S}$ satisfies the Fubini property from Definition \ref{Fubini property in the affine grassmanian}, which is equivalent to the graph of $\psi_{k}$ satisfying the Fubini property \eqref{Fubini property along V} along the horizontal direction $\pi_Y$. Thus, the result of Lemma \ref{fibers} still applies in this case, and by \eqref{Lipschits graph 2} we get
\begin{equation*}
  \hdim (F^{k+1}\setminus D)=\hdim \psi_{k}(F^{k+1}\setminus D)+ \esssup^\beta_{W\in  \mathcal{P}^k(F)}\hdim(\psi_{k}^{-1}(W)\cap (F^{k+1}\setminus D)),
 \end{equation*}  
with $\beta=\hdim(\mathcal{P}^k(F))$, by using $\psi_{k}(F^{k+1}\setminus D)=\mathcal{P}^k(F)$ and $D\subsetneq F^{k+1}$ lower dimensional, we can rewrite
\begin{equation*}
  \hdim (F^{k+1})=\hdim(\mathcal{P}^k(F))+ \esssup^\beta_{W\in \mathcal{P}^k(F)}\hdim(\psi_{k}^{-1}(W)\cap (F^{k+1}\setminus D))
 \end{equation*}  
or
\begin{equation}\label{Lemma: Fubini applied}
  \hdim(\mathcal{P}^k(F))=\hdim (F^{k+1})- \esssup^\beta_{W\in \mathcal{P}^k(F)}\hdim(\psi_{k}^{-1}(W)\cap (F^{k+1}\setminus D)).
 \end{equation}   
 Notice that for each $W\in \mathcal P^{k} (F)$ we have $$\psi_{k}^{-1}(W)\cap (F^{k+1}\setminus D)\subset(W\cap F)^{k+1}= W^{k+1}\cap F^{k+1},$$
 so, pointwise in $W$ we have $\hdim\left(\psi_{k}^{-1}(W)\cap( F^{k+1}\setminus D)\right)\leq \hdim\left(W^{k+1}\cap F^{k+1}\right)$ and thus 
 \begin{equation*}
    \esssup^\beta_{W\in \mathcal{P}^k(F)}\hdim\left(\psi_{k}^{-1}(W)\cap (F^{k+1}\setminus D)\right)\leq \esssup^\beta_{W\in \mathcal{P}^k(F)} \hdim\left(W^{k+1}\cap F^{k+1}\right).
\end{equation*}
Here we use the second assumption that says that $\displaystyle \esssup^\beta_{W\in \mathcal{P}^k(F)} \hdim\left(W^{k+1}\cap F^{k+1}\right)$ preserves the expected size of a generic section given by \eqref{Slicing E^{k+1} 2}, that is
\begin{equation*}
    \esssup^\beta_{W\in \mathcal{P}^k(F)} \hdim\left(W^{k+1}\cap F^{k+1}\right) \leq \max\left\{0,\hdim(F^{k+1})-(k+1)(d-k)\right\}.
\end{equation*}
Putting the two inequalities together into \eqref{Lemma: Fubini applied}, we get
\begin{align}\label{Main P^k inequality}
  \hdim(\mathcal{P}^k(F))&\geq \hdim (F^{k+1})-\max\left\{0,\hdim(F^{k+1})-(k+1)(d-k)\right\}\\ \nonumber
 &=\begin{cases} 
  \hdim (F^{k+1})   &   \text{if }\  \hdim(F^{k+1})\leq(k+1)(d-k) \\
  (k+1)(d-k)  & \text{if }\ \hdim(F^{k+1})>(k+1)(d-k)\ .
\end{cases}
\end{align}
Finally, since $\hdim (F^{k+1})\geq (k+1)\hdim(F)\geq(k+1)t$, the inequality \eqref{Main P^k inequality} implies that $\hdim(\mathcal{P}^k(E))\geq\hdim(\mathcal{P}^k(F))\geq (k+1)\min\{t,d-k\}$ finishing the proof.
\end{proof}
Previously, it has been mentioned that we believe that the non-concentration conditions imposed in Theorem \ref{k-affine span Fubini} are natural to consider, and not particularly restrictive in the context of the incidence set $\mathcal{S}$. For a general compact set $E\in \R^d$, it is certainly hard to verify or to discard the conditions. However, when the Hausdorff dimension of the set $E^{k+1}$ behaves well under cartesian products and intersections with $k$-planes, for example, if $E$ is Ahlfors regular or simply  $\hdim(E)=\dim_{\text{P}}(E)$, there are instances when the conditions of the theorem can be explicitly verified. Here, we present an example of a family of compact sets in $\R^d$ that satisfy the Fubini conditions imposed in Theorem \ref{k-affine span Fubini}.
\begin{example}[Angularly thin fractal shells]
\label{example:thin shells}
Fix $1\leq k\leq m$, and assume $d\geq (k+1)(m+1).$
Let $C\subset[1,2]$ be a compact Ahlfors $c$-regular set with $0<c<1$, and choose
mutually orthogonal linear $(m+1)$-subspaces
\[
L_0,\ldots,L_k\subset\G(d,m+1)
\]
Set $\displaystyle  K_i:=L_i\cap\mathbb S^{d-1}, \text{ and } K:=\bigcup_{i=0}^kK_i$ is the angular directions set, then $\hdim (K)=m<d-1$, define $F_i:=\{r\omega:r\in C,\ \omega\in K_i\}$ and finally $\displaystyle F=\bigcup_{i=0}^kF_i.$

By Ahlfors regularity of $C$ and $K_i\cong \mathbb{S}^m$,then each $F_i$ is bi-Lipschitz equivalent to $C\times\mathbb S^m$. Moreover, the sets $F_i$ are positively
separated, and hence $F$ is $(m+c)$-Ahlfors regular. In particular, $$\hdim(F)=\hdim(F)=m+c>k,
\qquad
\hdim(F^{k+1})=(k+1)(m+c).$$
We show that $\mathcal P^k(F)$ is lower dimensional in $\mathcal A(d,k)$ and that $F$ satisfies both hypotheses of
Theorem~\ref{k-affine span Fubini}. Remember the affine span map $\psi_k(\boldsymbol{x})=\aff\{x_0,\ldots,x_k\},$  defined on $F^{k+1}\setminus D$. Also, consider the sub collection $$\Omega:=F_0\times\cdots\times F_k\subset F^{k+1}.$$
Every tuple  $\boldsymbol{y}\in\Omega$ is made of orthogonal vectors. Therefore, the tuples in $\Omega$ are uniformly nondegenerate as can be seen by considering the matrix association $\M(\boldsymbol{y})$ as in \eqref{matrix column vectors for degeneracy} which is clear that has full rank for every $\boldsymbol{y}\in\Omega$. Then, just as in the proof of Theorem~\ref{k-affine span Fubini}, we obtain that $\psi_k|_{\Omega}$ is bi-Lipschitz onto its image. Therefore, for $\mathcal P^k(\Omega):=\psi_k(\Omega)\subset\mathcal P^k(F),$ we have $\hdim(P^k(\Omega))=(k+1)(m+c).$

On the other hand, Lemma~\ref{D is lower dimensional} gives $\hdim(\mathcal P^k(F))\leq\hdim(F^{k+1}\setminus D)=(k+1)(m+c).$ Thus,
\begin{equation}\label{thin shells Pk dimension}
\hdim(\mathcal P^k(F))=(k+1)(m+c)=:\beta.
\end{equation}
Observe that since $(m+c)<d-k$, then $\mathcal P^k(F)<\hdim(\A(d,k))$, so $\mathcal P^k(F)$ is not a subset of full dimension. Now we verify that $F$ satisfy the hypothesis of Theorem \ref{k-affine span Fubini}.

We first verify the Fubini hypothesis. By the countable Lipschitz decomposition of $S$ away from $D$ used in the proof of
Theorem \ref{k-affine span Fubini} given by  \eqref{countable union for S}, the incidence set $\mathcal S$ associated to $F$ satisfies
\begin{equation*}
\hdim(\mathcal S)=\hdim(F^{k+1}\setminus D)=(k+1)(m+c)=\beta.    
\end{equation*}
Since $\pi_{X}(\mathcal{S})=\mathcal{P}^{k}(F)$, the general Fubini inequality \eqref{Fubini always true} gives
\begin{equation}
 \beta+ \esssup^\beta_{W\in\mathcal P^k(F)} \hdim(\psi_k^{-1}(W))\leq\beta,   
\end{equation}
meaning that $\displaystyle \esssup^\beta_{W\in\mathcal P^k(F)} \hdim(\psi_k^{-1}(W))=0,$ and equality holds in the Fubini formula. Thus, the first hypothesis of Theorem \ref{k-affine span Fubini} is satisfied.

Now we verify the second hypothesis, for every
$W=\psi_k(x_0,\ldots,x_k)\in \mathcal{P}^k(\Omega),$ note that by orthogonality of the list of points $\{x_0,\dots,x_k\}$, we have $W\cap L_i={x_i}$ for each $L_i$, therefore $W\cap F=\{x_0,\dots ,x_k\}$ meaning that 
\begin{equation}\label{thin shells good sections}
\hdim
\bigl(F^{k+1}\cap W^{k+1}\bigr)=\hdim
\big((F\cap W)^{k+1}\big)=0.    
\end{equation}
Next we will show that the affine planes $(\mathcal{P}^k(F)\setminus\mathcal{P}^k(\Omega))$ form a
$\beta$-null family.

Given an affinely independent $(k+1)$-tuple of points in $F$, $\boldsymbol{x}\in F^{k+1}\setminus D$. Let
$n_i$ denote the number of its points belonging to $F_i$. This is, $n_i\geq 0$ and $\sum_{i=0}^kn_i=k+1.$ For example, if $\boldsymbol{x}\in \Omega$, we have $n_0=\cdots=n_k=1$, any other distribution has some $n_i\geq2$.

So, for any $\boldsymbol{x}\in (F^{k+1}\setminus D)\setminus\Omega$ there must be at least one $n_i\geq 2$, the $n_i$ points lie in $L_i$ and they are affinely independent and span a $(n_i-1)$-subplane contained in $L_i$.  The family of all such planes has dimension at most $$\dim(\mathcal A(m+1,n_i-1))= n_i(m+2-n_i) < n_i(m+c),$$ because $2-n_i\leq0$. It follows that every distribution other than $n_0=\cdots=n_k=1$ parametrizes a family of $k$-planes of Hausdorff dimension strictly less than
$$\sum_{i=0}^kn_i m=(k+1)m<\beta.$$
Therefore
\begin{equation}\label{thin shells exceptional family}
\dim_{\mathrm H}
\bigl(\mathcal P^k(F)\setminus\mathcal P^k(\Omega)\bigr)<\beta.
\end{equation}
Combining \eqref{thin shells good sections} and
\eqref{thin shells exceptional family} and under the identification $W\mapsto W^{k+1}\in\mathcal W^{k+1}(F)$ from \eqref{Subset W^k}, we obtain $$\displaystyle \esssup^\beta_{W\in\mathcal P^k(F)} \hdim \bigl(F^{k+1}\cap W^{k+1}\bigr)\ =\esssup^\beta_{W^{k+1}\in\mathcal W^{k+1}(F)} \hdim \bigl(F^{k+1}\cap W^{k+1}\bigr)=0.$$

Finally, by $\hdim(F^{k+1})=(k+1)(m+c)<(k+1)(d-k)$ we conclude $$\max\left\{ 0,\,\hdim(F^{k+1})-(k+1)(d-k)\right\}=0.$$
Then the second hypothesis of Theorem \ref{k-affine span Fubini} is also satisfied.
\end{example}
\begin{figure}[H]
    \centering
\includegraphics[height=5.5cm,width=6cm]{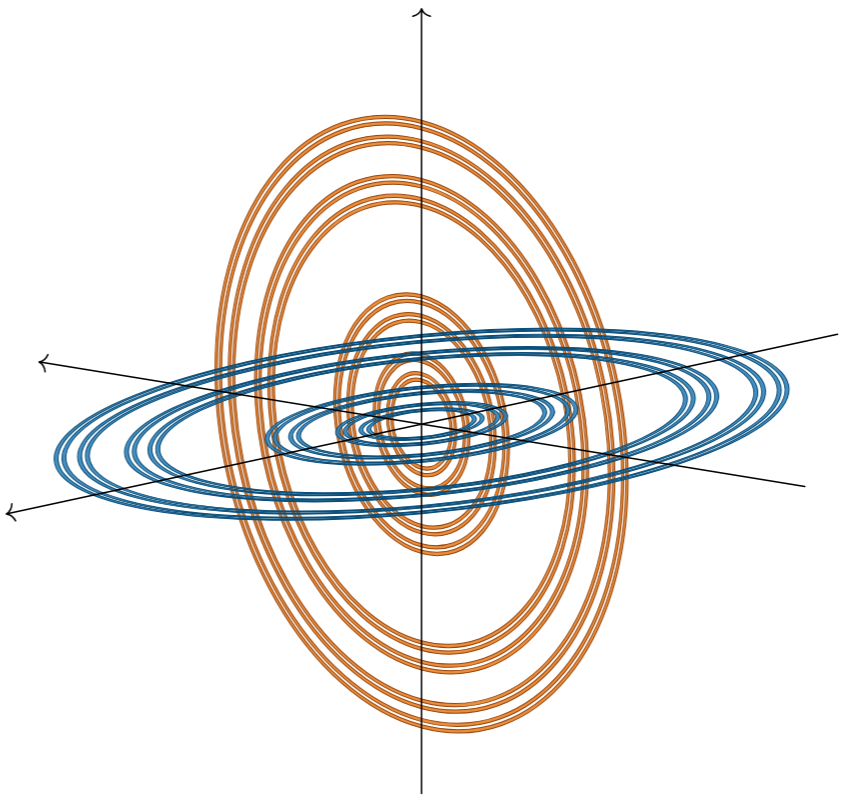}
    \caption{Representation of a set of the family described in Example \ref{example:thin shells}, in the picture we show  $F\subset \R^3$ with $m=1$ and $\hdim(F)=1+\frac{\ln(2)}{\ln(3)}<2$, the planes are $L_0: x=0$ and $L_1: z=0$, then the angular set is the union of a meridian and the equator $K=(L_0\cap \mathbb{S}^2)\cup(L_1\cap \mathbb{S}^2)$ and the radial shell is parametrized by $C$, a translated copy of the middle third Cantor set away from the origin.}
\end{figure}

Observe that the conditions imposed in Theorem \ref{k-affine span Fubini} are well adjusted for a set $E$ with $d_E=d$. In the case when the set is essentially concentrated inside some affine subspace we would have $d_E<d$. Thus, Corollary \ref{corollary: effective dimension} follows naturally after choosing $V\in\A(d,d_E)$ of minimal dimension satisfying $\eqref{effective dimension of E}$ and restricting the argument of the proof of Theorem \ref{k-affine span Fubini} to the set $E\cap V$, after noticing that $\hdim(E\cap V)=\hdim(E)$ and identifying $E \cap V\subset\R^{d_E}\times\{0\}^{d-d_E}\cong \R^{d_E}$. Using this Corollary, we can also draw conclusions about volumes of simplices, even though it is not an unconditional result as Theorem \ref{main}. 

We would like to remark that the following result about volumes of simplices could have been obtained by applying Theorem \ref{main} inside the affine subspace $V$, since $\hdim (E \cap V ) = \hdim (E)$. However, the value of this result as stated in terms of the hypothesis of Theorem \ref{k-affine span Fubini}, is that it provides an alternative approach to the volumes of simplices problem under additional assumptions, rather than being a stronger theorem than Theorem \ref{main}.
\begin{cor}\label{Volumes with Fubini}
Fix $k\geq 1$ and let $E\subset \R^d$ be a Borel set with $\hdim(E)>t>\max\{k,d_E-k \}$, 
and given $V\in\A(d,d_E)$ of minimal dimension satisfying $\eqref{effective dimension of E}$, let $F\subset E\cap V$ be compact such that $\hdim(F)\geq t$. After identifying $V\cong\R^{d_E}$, assume that the following
conditions hold:
\begin{enumerate}
    \item The incidence set $\mathcal{S}$ of $F$ as in $\eqref{Incidence set S}$ satisfies the \textbf{Fubini property} from Definition \ref{Fubini property in the affine grassmanian}.

        \item The inequality 
        \begin{equation}
        \esssup^\beta_{W^{k+1}\in\mathcal{W}^{k+1}(F)} \hdim(F^{k+1}\cap W^{k+1})\leq\max\left\{0,\hdim(F^{k+1})-(k+1)(d_E-k)\right\} 
        \end{equation}
         holds true, where $\beta=\hdim(\mathcal P^{k} (F))$ and $\mathcal{W}^{k+1}(F)$  is given by $\eqref{Subset W^k}$.
\end{enumerate}
 Then there exist $x_{0},\ldots,x_{k}\in E$ such that the set
  \[
    \Vol_{k+1}^{(x_0,\ldots,x_{k})}(E) :=\left\{ \Vol_{k+1}(x_0,\ldots,x_k,y): y\in E \right\}
  \]
has positive Lebesgue measure. 
\end{cor}
\begin{proof}
By the Corollary \ref{corollary: effective dimension} applied on the set $F$, we obtain $\hdim(\mathcal P^k(F))=(k+1)(d_E-k).$ Let $$ \mathcal P_{\dir}^{\perp}:=\{W_0^{\perp}\in\mathcal G(d_E,d_E-k):W=W_0+p\in\mathcal P^k(F)\}\subset\G(d_E,d_E-k),$$
the same argument as in \eqref{P_dir} proves $\hdim (\mathcal{P}^k(F)) \leq \hdim(\mathcal P_{\dir}^{\perp})+(d_E-k),$ and hence $\hdim(\mathcal P_{\dir}^{\perp})=k(d_E-k).$ We also choose some $u$ such that $\hdim(F)\geq t>u> \max\{k,d_E-k\}$, then we hare able to find a compact $K\subset F\subset\R^{d_E}$ such that $$0<\mathcal H^u(K)<\infty.$$ 

For such compact set, we apply Theorem \ref{exmar2} and obtain the estimates for $\Theta(u,d_E,d_E-k)$ and the threshold $u>d_E-k$ to conclude that $\hdim( \mathcal P_{\dir}^{\perp})>\hdim\Theta(u,d_E,d_E-k)$, the rest of the proof follows as in the proof of Theorem \ref{main}
\end{proof}

Lastly, we remark that although the concepts of minimal dimension for irreducibility of a measure $\mu$ supported on $E$, $d_\mu$, and the effective dimension $d_E$ of $E$ are closely related, they are different in general, see the following example:
\begin{example}
Let $L_1,L_2\subset \R^3$ be two nonparallel line segments contained in a plane, define $V:=\aff\{L_1,L_2\}$. Let $L_3\subset\R^3$ be another line segment not contained in $V$, and choose a compact set $C\subset L_3\setminus V$
such that $\hdim(C)=s<1.$
Consider $E:=L_1\cup L_2\cup C$, by construction $\hdim(E)=1.$ 

Note that $V=\aff\{L_1,L_2\}$ satisfies $\eqref{effective dimension of E}$ with minimal dimension, since $\hdim(E\setminus V)=s<1$. Moreover, for every affine line $L\subset\R^3$, at least one of the sets $L_1\setminus L$ or $L_2\setminus L$ has Hausdorff dimension $1$, this means $d_E=2$.

On the other hand, consider the measure $\mu=\mathcal{H}^1|_{E}$, denoting the $1$-dimensional Hausdorff measure restricted to $E$,  then $\mu(E\cap L_1)>0$. This means that the measure $\mu|_{E\cap V}$ is not irreducible on the plane $V$ because $\mu(V')>0$ for $V'=\aff(L_1)\subset V$ of dimension $1$; however, $\mu|_{E\cap V'}$ is irreducible  on $V'$, so $d_\mu=1$. This example shows that for a measure $\mu$, its minimal dimension for irreducibility does not equal $d_E$ in general.
\end{example}

\begin{figure}[H]
    \centering
    \includegraphics[scale=0.15]{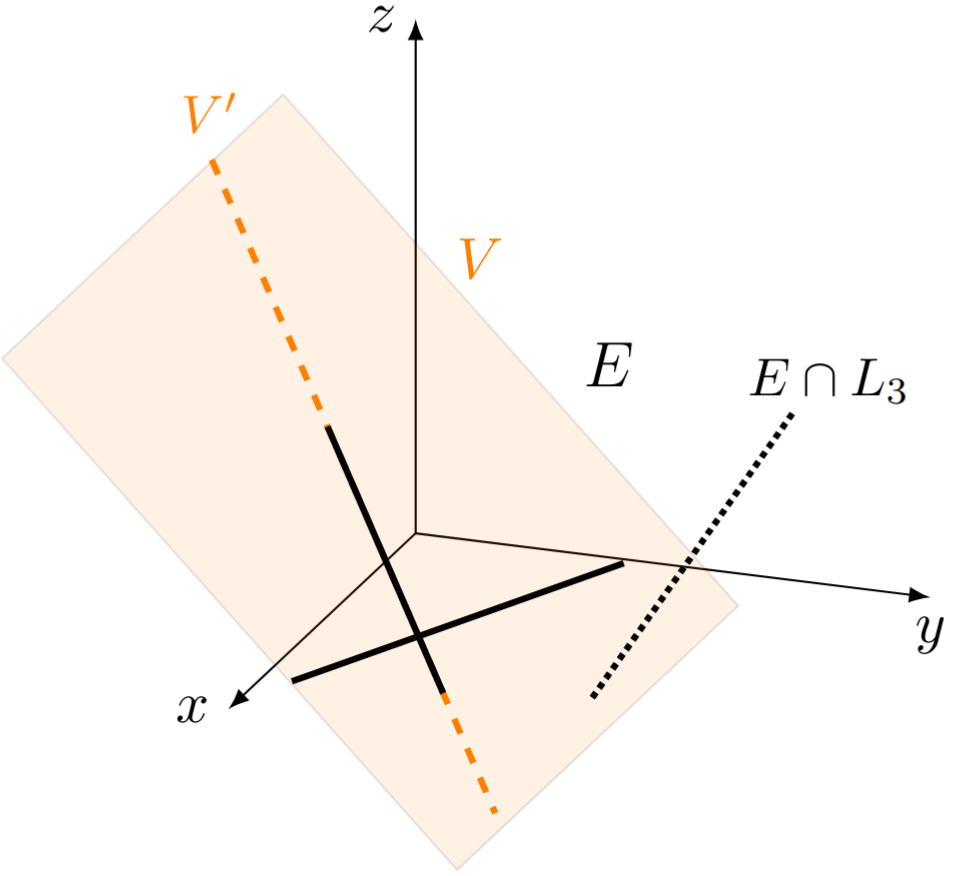}
    \caption{Set $E\subset \R^3$ with $d_E=2$, for which the measure $\mu=\mathcal H^1|_E$ has irreducible dimension $d_\mu=1$.}
    \label{fig:d_mu vs d_E 2}
\end{figure}

\end{document}